\documentclass[10pt, english]{amsart}

\usepackage{amsmath,amssymb,enumerate}

\usepackage[T1]{fontenc}
\usepackage[all]{xy}

\usepackage{babel}
\usepackage{amstext}
\usepackage{amsmath}
\usepackage{amsfonts}
\usepackage{latexsym}
\usepackage{ifthen}

\usepackage{xypic}
\xyoption{all}
\newcommand{\codim}{{\rm codim}}

\newcommand{\sJ}{{\mathcal J}}

\newtheorem{lemma1}{}[section]

\newenvironment{lemma}{\begin{lemma1}{\bf Lemma.}}{\end{lemma1}}
\newenvironment{example}{\begin{lemma1}{\bf Example.}\rm}{\end{lemma1}}

\newenvironment{theorem}{\begin{lemma1}{\bf Theorem.}}{\end{lemma1}}
\newenvironment{proposition}{\begin{lemma1}{\bf Proposition.}}{\end{lemma1}}
\newenvironment{corollary}{\begin{lemma1}{\bf Corollary.}}{\end{lemma1}}
\newenvironment{remark}{\begin{lemma1}{\bf Remark.}\rm}{\end{lemma1}}

\newenvironment{definition}{\begin{lemma1}{\bf Definition.}}{\end{lemma1}}

\newenvironment{conjecture}{\begin {lemma1}{\bf Conjecture.}}{\end{lemma1}}
\newenvironment{question}{\begin{lemma1}{\bf Question.}}{\end{lemma1}}
\newenvironment{problem}{\begin{lemma1}{\bf Problem.}}{\end{lemma1}}

\newenvironment{remark*}{{\bf Remark.}}{}
\newenvironment{example*}{{\bf Example.}}{}
\newenvironment{assumption*}{{\bf Assumption.}}{}

\newcommand{\R}{\ensuremath{\mathbb{R}}}
\newcommand{\Q}{\ensuremath{\mathbb{Q}}}
\newcommand{\Z}{\ensuremath{\mathbb{Z}}}

\newcommand{\N}{\ensuremath{\mathbb{N}}}
\newcommand{\PP}{\ensuremath{\mathbb{P}}}

\newcommand{\merom}[3]{\ensuremath{#1:#2 \dashrightarrow #3}}

\newcommand{\holom}[3]{\ensuremath{#1:#2  \rightarrow #3}}
\newcommand{\fibre}[2]{\ensuremath{#1^{-1} (#2)}}

\makeatletter
\ifnum\@ptsize=0  \fi \ifnum\@ptsize=2  \fi 

\newcommand\sF{{\mathcal F}}

\newcommand\sI{{\mathcal I}}

\newcommand\sO{{\mathcal O}}

\newcommand\bN{{\mathbb N}}

\DeclareMathOperator*{\Pic0}{Pic^0}

\DeclareMathOperator*{\supp}{Supp}

\newcommand{\Null}[1]{\ensuremath{\mbox{Null}(#1)}}
 
\DeclareMathOperator{\nd}{nd}
\DeclareMathOperator{\DIV}{div}

\DeclareMathOperator{\rank}{rank}

\newcommand{\NE}[1]{ \ensuremath{ \overline { \mbox{NE} }(#1)} }

\title{Rational curves on compact K\"ahler manifolds} 
\date{October 25, 2017}

\subjclass[2000]{32J27, 14E30, 14J35, 14J40, 14M22, 32J25}
\keywords{MMP, rational curves, K\"ahler manifolds, relative adjoint classes, subadjunction}

\author{Junyan Cao}
\author{Andreas H\"oring}

\address{Junyan Cao, Institut de Math\'ematiques de Jussieu, Universit\'e Pierre et Marie Curie, Case 247, 4 Place Jussieu, 75252 Paris Cedex, France}
\email{Junyan.CAO@imj-prg.fr}

\address{Andreas H\"oring, Laboratoire de Math{\'e}matiques J.A. Dieudonn{\'e},
UMR 7351 CNRS, Universit{\'e} de Nice Sophia-Antipolis, 06108 Nice Cedex 02, France        
}
\email{hoering@unice.fr}

\begin{document}

\begin{abstract} 
Mori's theorem yields the existence of rational curves on projective manifolds
such that the canonical bundle is not nef. In this paper we study 
compact K\"ahler manifolds such that the canonical bundle
is pseudoeffective, but not nef. We present an inductive argument for the
existence of rational curves that uses neither deformation theory nor reduction
to positive characteristic. 
The main tool for this inductive strategy is a weak subadjunction formula for lc centres associated
to certain big cohomology classes. 
\end{abstract}

\maketitle

\section{Introduction}

\subsection{Main results}

Rational curves have played an important role in the classification theory of projective manifolds ever since Mori showed
that they appear as a geometric obstruction to the nefness of the canonical bundle.

\begin{theorem} \cite{Mor79, Mor82}
Let $X$ be a complex projective manifold such that the canonical bundle $K_X$ is not nef. Then there exists a rational curve
$C \subset X$ such that $K_X \cdot C<0$.
\end{theorem}

This statement was recently generalised to compact K\"ahler manifolds of dimension three \cite{HP16},
but the proof makes crucial use of results on deformation theory of curves on threefolds which 
are not available in higher dimension. Mori's proof uses
a reduction to positive characteristic in an essential way and thus does not adapt to the more general analytic setting. 
The aim of this paper is to develop a completely different, inductive approach to the existence
of rational curves. Our starting point is the following

\begin{conjecture} \label{conjectureBDPP}
Let $X$ be a compact K\"ahler manifold. Then the canonical class $K_X$ is pseudoeffective if and only if $X$ is not uniruled (i.e.
not covered by rational curves).
\end{conjecture}

This conjecture is shown for projective manifolds in \cite{MM86, BDPP13} and it is also known in dimension three by a theorem of 
Brunella \cite{Bru06} using his theory of rank one foliations. Our main result is as follows:

\begin{theorem} \label{theoremmain}
Let $X$ be a compact K\"ahler manifold of dimension $n$. 
Suppose that Conjecture \ref{conjectureBDPP} holds 
for all manifolds of dimension at most $n-1$. 
If $K_X$ is pseudoeffective but not nef, there exists a $K_X$-negative rational curve $f: \PP^1 \rightarrow X$. 
\end{theorem}

Our statement is actually a bit more precise: the $K_X$-negative rational curve has zero intersection with 
a cohomology class that is nef and big, so the class of the curve lies in an extremal face of the (generalised) Mori cone.
Theorem \ref{theoremmain} is thus a first step towards a cone and contraction theorem for K\"ahler manifolds of arbitrary dimension. 

In low dimension we can combine our theorem with Brunella's result:

\begin{corollary} \label{corollarylowdimension}
Let $X$ be a compact K\"ahler manifold of dimension at most four. 
If $K_X$ is pseudoeffective but not nef, there exists a rational curve $f: \PP^1 \rightarrow X$ such that $K_X \cdot f(\PP^1)<0$.
\end{corollary}

\subsection{The strategy}

The idea of the proof is quite natural and inspired by well-known results of the minimal model program: let $X$ be a compact K\"ahler
manifold such that $K_X$ is pseudoeffective but not nef. We choose a K\"ahler class $\omega$ 
such that $\alpha := K_X+\omega$ is nef and big but not K\"ahler. If we suppose that $X$ is projective and $\omega$ is an $\R$-divisor
class we know by the base point free theorem \cite[Thm.7.1]{HM05} that there exists a morphism
$$
\mu : X \rightarrow X'
$$
such that $\alpha = \mu^* \omega'$ with $\omega'$ an ample $\R$-divisor class on $X'$. Since $\alpha$ is big the morphism
$\mu$ is birational, and we denote by $Z$ an irreducible component of its exceptional locus. A general fibre 
of $Z \rightarrow \mu(Z)$ has positive dimension and is covered by rational curves, in particular $Z$
is uniruled. More precisely, denote by $k \in \N$ the dimension of $\mu(Z)$. Since $\alpha= \mu^* \omega'$ we
have $(\alpha|_Z)^{k+1}=0$ and $(\alpha|_Z)^k$ is represented by some multiple of $F$
where $F$ is an irreducible component of a general fibre of
$Z \rightarrow \mu(Z)$. Since $F$ is an irreducible component of a $\mu$-fibre the conormal sheaf is ``semipositive'', so we expect that
\begin{equation} \label{naivecase}
K_{F'} \cdot \pi^* \omega|_F^{\dim Z-k-1} \leq  \pi^* K_X|_F \cdot \pi^* \omega|_F^{\dim Z-k-1}
\end{equation}
where $\holom{\pi}{F'}{F}$ is a desingularisation of $F$. Since $\alpha|_F$ is trivial and $K_X = \alpha - \omega$
we see that the right hand side is negative, in particular $K_{F'}$ is not pseudoeffective. 
Thus we can apply \cite{MM86, BDPP13} to $F'$ and obtain that $F$ is uniruled. 
Since $F$ is general we obtain that $Z$ is uniruled. The key idea of our approach is to prove a numerical analogue
of \eqref{naivecase} that does not assume the existence of the contraction.

Indeed if $X$ is K\"ahler we are far from knowing the existence of a contraction. However we can still consider the null-locus
$$
\Null{\alpha} = \bigcup_{\int_Z \alpha|_Z^{\dim Z}=0} Z.
$$
It is easy to see that if a contraction theorem holds also in the K\"ahler setting, then the null-locus is exactly the
exceptional locus of the bimeromorphic contraction. We will prove
that at least one of the irreducible components $Z \subset \Null{\alpha}$ is covered by $\alpha$-trivial rational curves: 
let $\holom{\pi}{Z'}{Z}$ be a desingularisation, and let
$k$ be the numerical dimension of $\pi^* \alpha|_Z$ (cf. Definition \ref{definitionnumericaldimension}). 
We will prove that
\begin{equation} \label{notnaivecase}
K_{Z'} \cdot \pi^* \alpha|_Z^{k} \cdot \pi^* \omega|_Z^{\dim Z-k-1} 
\leq 
\pi^* K_X|_Z \cdot \pi^* \alpha|_Z^{k} \cdot \pi^* \omega|_Z^{\dim Z-k-1}.
\end{equation}
Note that the right hand side is negative, so Conjecture \ref{conjectureBDPP} yields the existence of rational curves.
Recall also that if the contraction $\mu$ exists, then $\pi^* \alpha|_Z^{k}$ is a multiple of a general fibre, so 
this inequality is a refinement of \eqref{naivecase}.
The inequality \eqref{notnaivecase} follows from a more general weak subadjunction formula for maximal
lc centres (cf. Definition \ref{definitionmaxlc}) of the pair $(X, c \alpha)$ (for some real number $c>0$)
which we will explain in the next
section. The idea of seeing the irreducible components of the null locus as an lc centre for a suitably chosen pair is 
already present in Takayama's uniruledness of stable base loci \cite{Tak08}, in our case 
a recent result of Collins and Tosatti \cite[Thm.1.1]{CT13} and the work of Boucksom \cite{Bou04} yield this property
without too much effort.

While \eqref{notnaivecase} and Conjecture \ref{conjectureBDPP} imply immediately that $Z$ is uniruled it is a priori
not clear if we can choose the rational curves to be $K_X$-negative (or even $\alpha$-trivial):
for the simplicity of notation, let us suppose that $Z$ is smooth.  
If $Z$ was projective and $\alpha|_Z$ an $\R$-divisor class we could argue as in \cite[Prop.7.11]{HP16} using Araujo's description of the mobile cone \cite[Thm.1.3]{Ara10}.
In the K\"ahler case we need a new argument: let $Z \rightarrow Y$ be the MRC-fibration (cf. Remark \ref{remarkMRC})
and let $F$ be a general fibre.
Arguing by contradiction we suppose that $F$ is not covered by $\alpha$-trivial rational curves.
A positivity theorem for relative adjoint classes (Theorem \ref{theoremdirectimage2}) shows that 
$K_{Z/Y} + \alpha|_Z$ is pseudoeffective if $K_F+ \alpha|_F$ is pseudoeffective. 
Since $K_Y$ is pseudoeffective by Conjecture \ref{conjectureBDPP} this implies that 
$K_Z+\alpha|_Z$ is pseudoeffective, a contradiction to  \eqref{notnaivecase}.

Thus we are left to show that $K_F+ \alpha|_F$ is pseudoeffective, at least up to replacing 
$\alpha|_F$ by $\lambda \alpha|_F$ for some $\lambda \gg 0$. Since $\alpha|_F$ is not a rational cohomology class this is a non-trivial property related to
the Nakai-Moishezon criterion for $\R$-divisors by  Campana and Peternell \cite{CP90}.
Using the minimal model program for the projective manifold $F$ and Kawamata's bound on the length
of extremal rays \cite[Thm.1]{Kaw91} we overcome this problem in Proposition \ref{propositionboundedpseff}.

\subsection{Weak subadjunction}

Let $X$ be a complex projective manifold, and let $\Delta$ be an effective $\Q$-Cartier divisor on $X$ 
such that the pair $(X, \Delta)$ is log-canonical. Then there is a finite number of log-canonical centres associated
to $(X, \Delta)$ and if we choose $Z \subset X$ an lc centre that is minimal with respect to the inclusion, the
Kawamata subadjunction formula holds \cite{Kaw98} \cite[Thm1.2]{FG12}: the centre $Z$ is a normal variety and there exists a boundary divisor $\Delta_Z$ such that $(Z, \Delta_Z)$ is klt and
$$
K_Z + \Delta_Z \sim_\Q (K_X+\Delta)|_Z.
$$
If the centre $Z$ is not minimal the geometry is more complicated, however we can still find an {\em effective} $\Q$-divisor $\Delta_{\tilde Z}$
on the normalisation $\holom{\nu}{\tilde Z}{Z}$ such 
that\footnote{This statement is well-known to experts, cf. \cite[Lemma 3.1]{a24} for a proof.}
$$
K_{\tilde Z} + \Delta_{\tilde Z} \sim_\Q \nu^* (K_X+\Delta)|_Z.
$$
We prove a weak analogue of the subadjunction formula
for cohomology classes:

\begin{theorem} \label{theoremweaksubadjunction}
Let $X$ be a compact K\"ahler manifold, and let $\alpha$ be a cohomology class on $X$
that is a modified K\"ahler class (cf. Definition \ref{definitionmodified}). Suppose that $Z \subset X$ is a maximal lc centre of the pair $(X, \alpha)$, 
and let $\holom{\nu}{\tilde Z}{Z}$ be the normalisation. 
Then we have 
$$
K_{\tilde Z} \cdot \omega_1 \cdot \ldots \cdot \omega_{\dim Z-1}
\leq
\nu^* (K_X+\alpha)|_Z \cdot \omega_1 \cdot \ldots \cdot \omega_{\dim Z-1},
$$
where $\omega_1, \ldots, \omega_{\dim Z-1}$ are arbitrary nef classes on $\tilde Z$.
\end{theorem}

Our proof follows the strategy of Kawamata in \cite{Kaw98}: given a log-resolution $\holom{\mu}{\tilde X}{X}$  
and an lc place $E_1$ dominating $Z$ we want to use a canonical bundle formula 
for the fibre space $\mu|_{E_1}: E_1 \rightarrow \tilde Z$ to relate $\mu^* (K_X+\alpha)|_{E_1}$ and $K_{\tilde Z}$.
As in \cite{Kaw98} the main ingredient for a canonical bundle formula is the positivity theorem for relative adjoint classes Theorem
\ref{theoremdirectimage} which, together with Theorem \ref{theoremdirectimage2}, is the main technical contribution
of this paper. The main tool of the proofs of Theorem \ref{theoremdirectimage} and Theorem \ref{theoremdirectimage2} is the positivity of the fibrewise Bergman kernel which is established in \cite{BP08, BP10}. 
Since we work with lc centres that are not necessarily minimal the positivity result Theorem \ref{theoremdirectimage}
has to be stated for pairs which might not be (sub-)klt.  This makes the setup of the proof quite heavy, but similar to earlier arguments
(cf. \cite{BP10, Pau12b} and \cite{FM00, Tak06} in the projective case).

The following elementary example illustrates Theorem \ref{theoremweaksubadjunction} and
shows how it leads to Theorem \ref{theoremmain}:

\begin{example}
Let $X'$ be a smooth projective threefold, and let $C \subset X'$ be a smooth curve such that the normal bundle
$N_{C/X'}$ is ample. Let $\holom{\mu}{X}{X'}$ be the blow-up of $X'$ along $C$ and let $Z$ be the exceptional divisor.
Let $D \subset X'$ be a smooth ample divisor containing the curve $C$, and let $D'$ be the strict transform.

By the adjunction formula we have $K_Z = (K_X+Z)|_Z$, in particular it is not true that
$K_{Z} \cdot \omega_1
\leq
K_X|_Z \cdot \omega_1$
for every nef class $\omega_1$ on $Z$. Indeed this would imply that $-Z|_Z$ is pseudoeffective, hence $N_{C/X'}^*$ is pseudoeffective
in contradiction to the construction. However if we set $\alpha := \mu^* c_1(D)$, then $\alpha$ is nef and represented
by $\mu^* D = D' + Z$. Then the pair $(X, D'+Z)$ is log-canonical and $Z$ is a maximal lc centre. Moreover we have 
$$
K_Z \cdot \omega_1 = (K_X+Z)|_Z \cdot \omega_1 \leq (K_X+ D' + Z)|_Z \cdot \omega_1 = (K_X+ \alpha)|_Z \cdot \omega_1
$$
since $D'|_Z$ is an effective divisor. 

Now we set $\omega_1= \alpha|_Z$, then $\alpha|_Z \cdot \omega_1 = \alpha|_Z^2 =0$ since it is a pull-back from $C$. Since $K_X$ is anti-ample on the $\mu$-fibres we have
$$
K_Z \cdot \alpha|_Z = K_X|_Z \cdot \alpha|_Z < 0.
$$
Thus $K_Z$ is not pseudoeffective.
\end{example}

\subsection{Relative adjoint classes}

We now explain briefly the idea of the proof of Theorem \ref{theoremdirectimage} and Theorem \ref{theoremdirectimage2}. In view of the main results in \cite{BP08} and \cite{Pau12},
it is natural to ask the following question :
\begin{question}\label{mainpositquestion}
Let $X$ and $Y$ be two compact K\"ahler manifolds of dimension $m$ and $n$ respectively, and let $f: X\rightarrow Y$ be a surjective map with connected fibres.
Let $F$ be the general fiber of $f$.
Let $\alpha_X$ be a K\"ahler class on $X$ and let $D$ be a klt $\mathbb{Q}$-divisor on $X$ such that $c_1 (K_F) + [(\alpha_X +D)|_F]$ is a pseudoeffective class.
Is $c_1 (K_{X/Y}) + [\alpha_X +D]$ pseudoeffective ?
\end{question}
In the case 
$c_1 (K_F) + [(\alpha_X +D)|_F]$ is a K\"{a}hler class on $F$, \cite{Pau12, Gue16} confirm the above question by studying the variation of K\"{a}hler-Einstein metrics
(based on \cite{Sch12}). 
In our article, we confirm Question \ref{mainpositquestion} in two special cases: Theorem \ref{theoremdirectimage} and Theorem \ref{theoremdirectimage2} by using the positivity of the fibrewise Bergman kernel 
which is established in \cite{BP08, BP10}.
Let us compare our results to P\u aun's result \cite[Thm.1.1]{Pau12} on relative adjoint classes: 
while we make much weaker
assumptions on the geometry of pairs or the positivity of the involved cohomology classes we are always in a situation where
locally over the base we only have to deal with $\R$-divisor classes. Thus the transcendental character of the argument is
only apparent on the base, not along the general fibres. 

More precisely, in Theorem \ref{theoremdirectimage}, we add an additional condition that
$c_1 (K_{X/Y} +[\alpha_X +D])$ is pull-back of a $(1,1)$-class on $Y$ (but we assume that $D$ is sub-boundary). Then we can take a Stein cover $(U_i)$ of $Y$
such that $ (K_{X/Y} +[\alpha_X +D]) |_{f^{-1}(U_i)}$ is trivial on $f^{-1}(U_i)$. Therefore $[\alpha_X +D] |_{f^{-1}(U_i)}$ is a $\R$-line bundle on $f^{-1}(U_i)$. 
We assume for simplicity that $D$ is klt (the sub-boundary case is more complicated).
We can thus apply \cite{BP10} to every pair $(f^{-1}(U_i), K_{X/Y} +[\alpha_X +D])$. Since the fibrewise Bergman kernel metrics are defined fiber by fiber, by using 
$\partial\overline{\partial}$-lemma, we can glue the metrics together and Theorem \ref{theoremdirectimage} is thus proved.

In Theorem \ref{theoremdirectimage2}, we add the condition that $F$ is simply connected and $H^0 (F, \Omega_F ^2) =0$
\footnote{If $F$ is rational connected these two conditions are satisfied.}. Then we can find a Zariski open set $Y_0$ of $Y$
such that $R^i f_* (\mathcal{O}_X) =0$ on $Y_0$ for every $i=1,2$. By using the same argument as in Theorem \ref{theoremdirectimage}, we can construct a quasi-psh function
$\varphi$ on $f^{-1} (Y_0)$ such that $\frac{\sqrt{-1}}{2\pi}\Theta (K_{X/Y}) + \alpha_X +dd^c \varphi \geq 0$ on $f^{-1} (Y_0)$. Now the main problem is to extend $\varphi$ 
to be a quasi-psh function on $X$. Since $c_1 (K_F +\alpha_X |_F)$ is not necessary a K\"{a}hler class on $F$, we cannot use directly the method in \cite[3.3]{Pau12} .
Here we use the idea in \cite{Lae02}. In fact, thanks to 
\cite[Part II, Thm 1.3]{Lae02}, we can find an increasing sequence $(k_m)_{m\in \N}$ and hermitian line bundles $(F_m, h_m)_{m\in \N}$ (not necessarily holomorphic) on $X$ such that
\begin{equation}\label{introconstr}
\|\frac{\sqrt{-1}}{2\pi} \Theta_{h_m} (F_m) - k_m (\frac{\sqrt{-1}}{2\pi}\Theta (K_{X/Y}) + \alpha_X)\|_{C^{\infty} (X)}\rightarrow 0 .
\end{equation}
Let $X_y$ be the fiber over $y\in Y_0$. As we assume that $H^0 (X_y, \Omega_{X_y} ^2) =0$, $F_m |_{X_y}$ can be equipped with a holomorphic structure $J_{X_y, m}$. 
Therefore we can define the Bergman kernel metric associated to $( F_m |_{X_y} , J_{X_y, m}, h_m)$.
Thanks to $\partial\overline{\partial}$-lemma, we can compare $\varphi|_{X_y}$ and the Bergman kernel metric associated to $( F_m |_{X_y} ,J_{X_y, m},  h_m)$.
Note that \eqref{introconstr} implies that $F_m$ is more and more holomorphic. Therefore, by using standard Ohsawa-Takegoshi technique \cite{BP10}, 
we can well estimate the Bergman kernel metric associated to $F_m |_{X_y}$ when $y\rightarrow Y\setminus Y_0$. Theorem \ref{theoremdirectimage2} is thus proved by combining these two facts.

{\bf Acknowledgements.} 
This work was partially supported by the A.N.R. project CLASS\footnote{ANR-10-JCJC-0111}.

\section{Notation and terminology}

For general definitions we refer to \cite{Har77, Kau83, Dem12}.
Manifolds and normal complex spaces will always be supposed to be irreducible.
A fibration is a proper surjective map with connected fibres \holom{\varphi}{X}{Y} between normal complex spaces.

\begin{definition} \label{definitiondecomp}
Let $X$ be a normal complex space, and let $\holom{f}{X}{Y}$ be a proper surjective morphism. A $\Q$-divisor $D$ is
$f$-vertical if $f(\supp D) \subsetneq Y$. Given a $\Q$-divisor $D$ it admits a unique decomposition
$$
D = D_{f\mbox{\tiny \rm -hor}} + D_{f\mbox{\tiny \rm -vert}}
$$
such that $D_{f\mbox{\tiny \rm -vert}}$ is $f$-vertical and every irreducible component $E \subset \supp D_{f\mbox{\tiny \rm -hor}}$ surjects onto $Y$.
\end{definition}

\begin{definition} \label{definitionc1}
Let $X$ be a complex manifold, and let $\sF$ be a sheaf of rank one on $X$ that is locally free in codimension one. 
The bidual $\sF^{**}$ is reflexive of rank one, so locally free, and
we set $c_1(\sF):=c_1(\sF^{**})$.
\end{definition}

Throughout this paper we will use positivity properties of real cohomology classes 
of type $(1,1)$, that is elements of the vector space $H^{1,1}(X) \cap H^2(X, \R)$.
The definitions can be adapted to the case of a normal compact K\"ahler space $X$
by using Bott-Chern cohomology for $(1,1)$-forms with local potentials \cite{HP16}.
In order to simplify the notation we will use the notation 
$$
N^1(X) := H^{1,1}(X) \cap H^2(X, \R).
$$
Note that for the purpose of this paper
we will only use cohomology classes that are pull-backs of nef classes on some smooth space, 
so it is sufficient to give the definitions in the smooth case.

\begin{definition} \label{definitionnef} \cite[Defn 6.16]{Dem12}
Let $(X, \omega_{X} )$ be a compact K\"ahler manifold, and let $\alpha \in N^1(X)$. 
We say that $\alpha$ is nef if for every $\epsilon> 0$, there is a smooth $(1,1)$-form $\alpha_{\epsilon}$
in the same class of $\alpha$ such that $\alpha_{\epsilon}\geq -\epsilon\omega_{X}$.

We say that $\alpha$ is pseudoeffective if there exists a $(1, 1)$-current $T\geq 0$ in the same class of $\alpha$.
We say that $\alpha$ is big if there exists a $\epsilon> 0$ such that $\alpha-\epsilon \omega_{X}$ is pseudoeffective.
\end{definition}

\begin{definition}
Let $X$ be a compact K\"ahler manifold, and let $\alpha \in N^1(X)$ be a nef and big cohomology class on $X$.
The null-locus of $\alpha$ is defined as
$$
\Null{\alpha} = \bigcup_{\int_Z \alpha|_Z^{\dim Z}=0} Z.
$$
\end{definition}

\begin{remark*}
A priori the null-locus is a countable union of proper subvarieties of $X$. However by \cite[Thm.1.1]{CT13}
the null-locus coincides with the non-K\"ahler locus $E_{nK}(\alpha)$, in particular it is an analytic subvariety of $X$.
\end{remark*}

\begin{definition} \cite[Defn 6.20]{Dem12} \label{definitionnumericaldimension}
Let $X$ be a compact K\"ahler manifold, and let $\alpha \in N^1(X)$ be a nef class.
We define the numerical dimension of $\alpha$ by 
$$
\nd (\alpha) :=
\max \{k \in \N \ | \ \alpha^{k}\neq 0 \mbox{ in } H^{2k}(X,\mathbb{R})\}.
$$
\end{definition}

\begin{remark} \label{remarknumericaldimension} 
A nef class $\alpha$ is big if and only if $\int_X \alpha^{\dim X} >0$ \cite[Thm.0.5]{DP04} 
which is of course equivalent to $\nd(\alpha)=\dim X$. 

By \cite[Prop 6.21]{Dem12} the cohomology class $\alpha^{\nd (\alpha)}$ can be represented 
by a non-zero closed positive $(\nd (\alpha),\nd (\alpha))$-current $T$.
Therefore 
$\int_X \alpha^{\nd (\alpha)}\wedge\omega_{X}^{\dim X - \nd (\alpha)} > 0$ for any K\"ahler class $\omega_{X}$.
\end{remark}

\begin{definition} \label{definitionintersection}
Let $X$ be a normal compact complex space of dimension $n$, and let $\omega_1, \ldots, \omega_{n-1} \in N^1(X)$
be cohomology classes. Let $\sF$ be a reflexive rank one sheaf on $X$, and let $\holom{\pi}{X'}{X}$ be a desingularisation.
We define the intersection number $c_1(\sF) \cdot \omega_1 \cdot \ldots \cdot \omega_{n-1}$ by 
$$
c_1((\mu^* \sF)^{**}) \cdot \mu^* \omega_1 \cdot \ldots \cdot \mu^* \omega_{n-1}.
$$
\end{definition}

\begin{remark*} The definition above does not depend on the choice of the resolution $\pi$:
the sheaf $\sF$ is reflexive of rank one, so locally free on the smooth locus of $X$. Thus $\mu^* \sF$ is locally free in the complement
of the $\mu$-exceptional locus. Thus $\holom{\pi_1}{X_1'}{X}$ and $\holom{\pi_2}{X_2'}{X}$ are two resolutions and $\Gamma$ is a manifold dominating $X_1'$ and $X_2'$ via bimeromorphic morphisms $q_1$ and $q_2$, then $q_1^* \pi_1^* \sF$
and $q_2^* \pi_2^* \sF$ coincide in the complement of the $\pi_1 \circ q_1=\pi_2 \circ q_2$-exceptional locus. Thus their
biduals coincide in the complement of this locus. By the projection formula their intersection with classes coming from $X$ are the same.
\end{remark*}

\section{Positivity of relative adjoint classes, part 1}

Before the proof of the main theorem in this section, we first recall the construction of fibrewise Bergman kernel metric and its important property, 
which are established in the works \cite{BP08,BP10}.
The original version \cite{BP10} concerns only the projective fibration. However, thanks to the optimal extension theorem \cite{GZ15}
and an Ohsawa-Takegoshi extension theorem for K\"{a}hler manifolds \cite{Yi14,Cao14}, we know that it is also true for the 
K\"{a}hler case :

\begin{theorem} {\cite[Thm 0.1]{BP10}, \cite[3.5]{GZ15}, \cite[Thm 1.1]{Yi14}\cite[Thm 1.2]{Cao14}}\label{maintool}
Let $p : X\rightarrow Y$ be a proper fibration between K\"{a}hler manifolds of dimension $m$ and $n$ respectively, and let $L$ be a line bundle endowed with a metric $h_L$ such that:

1) The curvature current of the bundle $(L, h_L)$ is semipositive in the sense of current, i.e., $\sqrt{-1}\Theta_{h_L} (L) \geq 0 ;$

2) there exists a general point $z\in Y$ and a non zero section $u\in H^0 (X_z, m K_{X_z} +L)$ such that
\begin{equation}\label{integralcond}
\int_{X_z} |u|_{h_L} ^{\frac{2}{m}} < +\infty . 
\end{equation}
Then the line bundle $m K_{X/Y} +L$ admits a metric with positive curvature current. Moreover, this metric is equal to the fibrewise $m$-Bergman kernel metric on the general fibre of $p$.
\end{theorem}

\begin{remark}\label{bergmankernelconstr}
Here are some remarks about the above theorem. 

{\em (1):} Note first that as $u\in H^0 (X_z, m K_{X_z} +L)$, $|u|_{h_L} ^{\frac{2}{m}}$ is a volume form on $X_z$.
Therefore the integral \eqref{integralcond} is well defined.

{\em (2):} The fibrewise $m$-Bergman kernel metric is defined as follows : Let $x\in X$ be a point on a smooth fibre of $p$. We first define a hermitian metric $h$ on $- (m K_{X/Y} +L)_x$
by
$$\|\xi\|_h ^2 :=\sup \frac{|\tau (x) \cdot \xi|^2}{(\int_{X_{p (x)}} |\tau|_{h_L} ^{\frac{2}{m}})^m} ,$$
where $\xi$ is a basis of  $- (m K_{X/Y} +L)_x$ and the 'sup' is taken over all sections $\tau\in H^0 (X_{p(x)} , m K_{X/Y} +L)$. The fibrewise $m$-Bergman kernel metric on $m K_{X/Y}+L$ is defined to be the dual of $h$.

It will be useful to give a more explicit expression of the Bergman kernel type metric.
Let $\omega_X$ and $\omega_Y$ be K\"{a}hler metrics on $X$ and $Y$ respectively. Then $\omega_X$ and $\omega_Y$
induce a natural metric $h_{X/Y}$ on $K_{X/Y}$. 
Let $Y_0$ be a Zariski open set of $Y$ such that $p$ is smooth over $Y_0$.
Set $h_0 := h_{X/Y} ^m \cdot h_L$ be the induced metric on $m K_{X/Y} +L$.
Let $\varphi$ be a function on $p^{-1} (Y_0)$ defined by 
$$\varphi (x) =\sup_{\tau \in A} \frac{1}{m} \ln |\tau|_{h_0} (x) ,$$
where
$$
A := \{f \ | \   f\in H^0 (X_{p(x)} , m K_{X/Y} +L)  \text{ and } \int_{X_{p (x)}} |f|_{h_0} ^{\frac{2}{m}} (\omega_X ^m / p^*\omega_Y ^n ) =1 \}.
$$
We can easily check that the metric $ h_0 \cdot e^{- 2 m \varphi}$ on $m K_{X/Y}+L$ coincides with the fibrewise $m$-Bergman kernel metric defined above.
In particular, $ h_0 \cdot e^{- 2 m \varphi}$ is independent of the choice of the metrics $\omega_X$ and $\omega_Y$.
Sometimes we call $\varphi$ the fibrewise $m$-Bergman kernel metric.

{\em (3):} Note that, by construction, if we replace $h_L$ by $f^\star c(y) \cdot h_L$ for some smooth strictly positive function $c(y)$ on $Y$, 
the corresponding weight function $\varphi$ in unchanged.
\end{remark}

For readers' convenience, we recall also the following version of the Ohsawa-Takgoshi extension theorem which will be used in the article.
\begin{proposition}\cite[Prop 0.2]{BP10}\label{keyextension}
Let $p : X\rightarrow \Delta$ be a fibration from a K\"{a}hler manifold to the unit disc $\Delta \in \mathbb{C}^n$.
and let $L$ be a line bundle endowed with a possible singular metric $h_L$ such that $\sqrt{-1}\Theta_{h_L} (L) \geq 0$ in the sense of current.
Let $m\in\mathbb{N}$.
We suppose that the center fiber $X_0$ is smooth and let $f\in H^0 (X_0, m K_{X_0} +L)$ such that
$$
\int_{X_0} |f|_{h_L} ^{\frac{2}{m}} < +\infty . 
$$
Then there exists a $F\in H^0 (X, m K_{X/Y} +L)$ such that 
\begin{enumerate}
\item[\rm (i)] $F|_{X_0} =f$
\item[\rm (ii)] The following $L^{\frac{2}{m}}$ bound holds
$$\int_{X} |F|_{h_L} ^{\frac{2}{m}} \leq C_0 \int_{X_0} |f|_{h_L} ^{\frac{2}{m}} .$$
where $C_0$ is an absolute constant as in the standard Ohsawa-Takegoshi theorem.
\end{enumerate}

Moreover, thanks to \cite{GZ15}, we can take $C_0$ as the volume of the unit disc $\Delta$.
\end{proposition}

Here is the main theorem in this section.

\begin{theorem}\label{theoremdirectimage}
Let $X$ and $Y$ be two compact K\"ahler manifolds of dimension $m$ and $n$ respectively, and let $f: X\rightarrow Y$ be a surjective map with connected fibres.
Let $\alpha_X$ be a K\"ahler class on $X$. Let\footnote{The somewhat awkward notation
will be become clear in the proof of Theorem \ref{theoremweaksubadjunction}.} $D=\sum\limits_{j=2}^k -d_j D_j$ be a $\mathbb{Q}$-divisor on $X$ such that the support has simple normal crossings.
Suppose that the following properties hold:

\begin{enumerate}
\item[\rm (a)] If $d_j \leq -1$ then $f(D_j)$ has codimension at least $2$. 
\item[\rm (b)] The direct image sheaf $f_* \mathcal{O}_X (\lceil-D\rceil)$ has rank one. Moreover, if $D= D^h~+~D^v$ is the decomposition in a $f$-horizontal part $D^h$ (resp. $f$-vertical part $D^v$)
then we have $(f_* \mathcal{O}_X (\lceil-D ^v\rceil))^{**} \simeq \mathcal{O}_Y$.
\item[\rm (c)] $c_1 (K_{X/Y} +\alpha_X + D)=f^* \beta$ for some real class $\beta\in H^{1,1} (Y, \mathbb{R})$. 
\end{enumerate}

Let $\omega_1 , \omega_2 , \cdots , \omega_{\dim Y -1}$ be nef classes on $Y$. Then we have
\begin{equation}\label{positivecondtion}
\beta \cdot \omega_1\cdots \omega_{\dim Y -1} \geq 0 .
\end{equation}
\end{theorem}

\begin{proof}

{\em Step 1: Preparation.}

We start by interpreting the conditions $(a)$ and $(b)$ in a more analytic language.
We can write the divisor $D$ as
$$
D= B - F^v - F^h,
$$ 
where $B, F^v, F^h $ are effective $\Q$-divisors and $F^v$ (resp. $F^h$) is $f$-vertical (resp. $f$-horizontal).
We also decompose $F^v$ as 
$$
F^v  = F_1 ^v + F_2 ^v
$$
such that $\codim_Y f( F_2 ^v) \geq 2$ and $\codim_Y f(E) = 1$ for every irreducible component $E \subset F_1^v$.

Let $X_y$ be a general $f$-fibre. Since $d_j > -1$ for every $D_j$ mapping onto $Y$ (cf. condition (a)),
the divisors $\lceil-D\rceil$ and  $\lceil F^h \rceil$ coincide over a non-empty Zariski open subset of $Y$.
Thus the condition $\rank f_* \mathcal{O}_X (\lceil-D\rceil)=1$ implies that 
$$
h^0 (X_y, \lceil F^h \rceil |_{X_y})=1.
$$
Therefore, for any meromorphic function $\zeta$ on $X_y$, we have
\begin{equation}\label{importantimply}
\DIV (\zeta) \geq - \lceil F^h \rceil |_{X_y} \qquad\Rightarrow \qquad \zeta \text{ is constant.}
\end{equation}
Since $d_j > -1$ for every $D_j$ mapping onto a divisor in $Y$ (cf. condition (a)),
the divisors $\lceil -D ^v\rceil$ and  $\lceil F^v \rceil$ coincide over a Zariski open subset $Y_1 \subset Y$
such that $\codim_Y (Y\setminus Y_1) \geq 2$.
In particular the condition $(f_* \mathcal{O}_X (\lceil-D ^v\rceil))^{**} \simeq \mathcal{O}_Y$ 
implies that $(f_* \mathcal{O}_X (\lceil-D ^v\rceil))|_{Y_1} = \sO_{Y_1}$.
So for every meromorphic function $\zeta$ on any small Stein open subset of $U \subset Y_1$, we have
\begin{equation}\label{importantimply2}
\DIV (\zeta \circ f) \geq - \lceil F^v \rceil|_{\fibre{f}{U}} \qquad\Rightarrow \qquad \zeta \text{ is holomorphic.}
\end{equation}

\medskip

{\em Step 2: Stein cover.} 

Select a Stein cover $( U_i )_{i\in I}$ of $Y$ such that $H^{1,1} (U_i ,\R)= 0$ for every $i$. 
Let $\theta$ be a smooth closed $(1,1)$-form in the same class of $c_1 (K_{X/Y} +\alpha_X +D + \lceil F^v + F^h \rceil)$.

Thanks to $(c)$, we have $c_1( K_{X/Y}+\alpha_X +D) |_{f^{-1} (U_i)} \in f^{-1} (H^{1,1} (U_i, \R)) =0$.
There exists thus a line bundle $L_i$ on $f^{-1} (U_i)$ such that 
$K_{X/Y} +L_i \simeq \lceil F^v + F^h \rceil$ on $f^{-1} (U_i)$.
Moreover, we can find a smooth hermitian metric $h_i$
on $K_{X/Y} + L_i$ over $f^{-1} (U_i)$ such that
\begin{equation}\label{12}
\frac{\sqrt{-1}}{2\pi} \Theta_{h_i} (K_{X/Y} +L_i) =\theta  \qquad\text{on }f^{-1} (U_i).
\end{equation}

\medskip

{\em Step 3: Local construction of metric.}

We construct in this step a canonical function $\varphi_i$ on $f^{-1} (U_i)$ such that
\begin{equation}\label{new4}
\theta +dd^c \varphi_i \geq  \lceil F_1 ^v +F^h \rceil \qquad\text{over }f^{-1} (U_i) \qquad\text{for every }i .
\end{equation}
The function is in fact just the potential of 
the fibrewise Bergman kernel metric mentioned in Remark \ref{bergmankernelconstr}. A more explicit construction is as follows:

Note first that $c_1 (L_i) =\alpha_X +D +\lceil F^v +F^h \rceil$, we can find a metric $h_{L_i}$
on $L_i$ such that 
$$i\Theta_{h_{L_i}} = \alpha_X +[D] +\lceil F^v +F^h \rceil =\alpha_X + [B] + (\lceil F^v +F^h \rceil - [F^v +F^h])\geq 0$$
in the sense of current. Moreover, we can ask that $ h_i /h_{L_i}$ is a global metric on $K_{X/Y}$, i.e., $h_i /h_{L_i}= h_j /h_{L_j}$
on $f^{-1} (U_i \cap U_j)$.

Thanks to the sub-klt condition $(a)$ and the construction of the metric $h_{L_i}$, we can find a Zariski open subset $U_{i,0}$ of $U_i$
such that for every $y\in U_{i,0}$, $f$ is smooth over $y$ and there exists a $s_y \in H^0 (X_y , K_{X/Y} +L_i)$ such that
\begin{equation}\label{adduniquesection}
\int_{X_y} |s_y |_{h_{L_i}} ^{2}  =1 .
\end{equation}
Recall that $|s_y |_{h_{L_i}} ^{2}$ is a volume forme on $X_y$ (cf. Remark \ref{bergmankernelconstr}).
Using the fact that 
\begin{equation}\label{addnew1}
h^0 (X_y, K_{X/Y} +L_i)= h^0 (X_y, \lceil F^h \rceil)=1  \qquad\text{for every }y\in U_{i,0},
\end{equation}
we know that $s_y$ is unique after multiplying by a unit norm complex number.
There exists thus a unique function $\varphi_i$ on $f^{-1} (U_{i,0})$ such that its restriction on 
$X_y$ equals to $\ln |s_y |_{h_i} $. We have the following key property.

\noindent\textbf{Claim:} $\varphi_i$ can be extended to be a quasi-psh function (we still denote it as $\varphi_i$) on $f^{-1} (U_i)$,
and satisfies \eqref{new4}.

The claim will be proved by using the methods in \cite[Thm 0.1]{BP08}. 
We postpone the proof of the claim later and first finish the proof of the theorem. The properties \eqref{importantimply}
and \eqref{importantimply2} will be used in the proof of the claim.

\medskip

{\em Step 4: Gluing process, final conclusion.}

We first prove that 
\begin{equation}\label{coincide}
\varphi_i=\varphi_j \qquad\text{on }f^{-1} (U_i \cap U_j) .
\end{equation}

Let $y\in U_{i,0} \cap U_{j,0}$. 
Since both $(K_{X/Y} +L_i)|_{X_y} \simeq (K_{X/Y} +L_j)|_{X_y} \simeq \lceil F^v +F^h \rceil |_{X_y}$,
we have $L_i |_{X_y} \simeq L_j |_{X_y}$. 
Under this isomorphism, the curvature condition \eqref{12} and $\partial\overline{\partial}$-lemma imply that 
\begin{equation}\label{new3}
h_{L_i} |_{X_y} =h_{L_j} |_{X_y} \cdot e^{-c_y} \qquad\text{for some constant }c_y \text{ on }X_y ,
\end{equation}
where the constant $c_y$ depends on $y\in Y$.
As $h_i /h_{L_i}$ is a metric on $K_{X/Y}$ independent of $i$, we have
\begin{equation}\label{new3cor}
h_{i} |_{X_y} =h_{j} |_{X_y} \cdot e^{-c_y} \qquad \text{ on }X_y .
\end{equation}
By \eqref{addnew1}, there exist unique elements
$s_{y, i} \in H^0 (X_y, K_{X/Y} +L_i)$ and $s_{y, j} \in H^0 (X_y, K_{X/Y} +L_j)$ (after multiply by a unit norm complex number) 
such that 
$$\int_{X_y} |s_{y,i} |_{h_{L_i}} ^2  =1 \qquad\text{and}\qquad\int_{X_y} |s_{y,j} |_{h_{L_j}} ^2   =1 .$$
Thanks to \eqref{new3}, we have (after multiply by a unit norm complex number) 
$$s_{y,i} =e^{\frac{c_y}{2}} \cdot s_{y , j} .$$
Together with \eqref{new3cor}, we get
\begin{equation}\label{fibrewiseestimate}
\varphi_i |_{X_y} = \ln |s_{y,i} |_{h_i}  = \ln |s_{y,j} |_{h_j} = \varphi_j |_{X_y} . 
\end{equation}
Since \eqref{fibrewiseestimate} is proved for every $y\in U_{i,0} \cap U_{j,0} $, we have $\varphi_i  = \varphi_j$ on $ f^{-1} (U_{i,0} \cap U_{j,0})$.
Combining this with the extension property of quasi-psh functions, 
\eqref{coincide} is thus proved.

Thanks to \eqref{coincide}, $( \varphi_i )_{i\in I}$ defines a global quasi-psh function on $X$ which we denote by $\varphi$.
By \eqref{new4}, we have 
$$\theta +dd^c \varphi \geq  \lceil F_1 ^v + F^h \rceil  \qquad\text{ over }f^{-1} (U_i)\qquad \text{for every }i.$$ 

Therefore 
$$\theta +dd^c \varphi \geq  \lceil F_1 ^v +F^h \rceil  \qquad\text{ over } X .$$ 
Then $c_1 (K_{X/Y} +\alpha_X +D +\lceil F_2 ^v\rceil)$ is pseudoeffective on $X$.
Together with the fact $\codim_Y f_* (F_2 ^v ) \geq 2$, the theorem is proved.
\end{proof}

The rest part of this section is devoted to the proof of the claim in Theorem \ref{theoremdirectimage}. The main method is 
the Ohsawa-Takegoshi extension techniques used in \cite{BP10}.
Before the proof of the claim, we need the following lemma which interprets the property \eqref{importantimply2} in terms of a condition on the metric $h_i$.

\begin{lemma}\label{uniformbound}
Fix a K\"ahler metric $\omega_X$ (resp. $\omega_Y$) on $X$ (resp. $Y$).
Let $s_B$ (resp. $s_{F^v}, s_{F^h}$) be the canonical section of the divisor $B$ (resp. $F^v$ and $F^h$).  Let $\psi$ be the function of the form
\begin{equation}\label{constructionsingularmetric}
\psi = \ln |s_B| -\ln |s_{F^v}| -\ln |s_{F^h}| + C^{\infty} ,
\end{equation}
where $|\cdot|$ is with respect to some smooth metric on the corresponding line bundle.
Let $Y_1$ be the open set defined in Step 1 of the proof of Theorem \ref{theoremdirectimage} and 
let $Y_0 \subset Y_1$ be a non-empty Zariski open set satisfying the following conditions :
\begin{enumerate}
\item[\rm (a)] $f$ is smooth over $Y_0$;
\item[\rm (b)] $f(D^v) \subset Y\setminus Y_0$;
\item[\rm (c)] $F^h|_{X_y}$ is snc for every $y \in Y_0$; 
\item[\rm (d)] The property \eqref{importantimply} holds for every $y\in Y_0$.
\end{enumerate}
Then for any open set $\Delta \Subset Y_1 \cap U_i$ (i.e., the closure of $\Delta$ is in $Y_1\cap U_i$), 
there exists some constant $C (\Delta, Y_1 ,U_i) >0$ depending only on $\Delta$, $Y_1$ and $U_i$, 
such that 
\begin{equation}\label{uniformestimate}
\int_{X_y} e^{-2\psi} \omega_X ^m /f^*\omega_Y ^n \geq C (\Delta, Y_1, U_i ) \qquad\text{for every }y \in \Delta \cap Y_0 ,
\end{equation}
where $m$ (resp. $n$) is the dimension of $X$ (resp. $Y$).
\end{lemma}

\begin{remark}\label{importantremark}
The meaning of \eqref{uniformestimate} is that, for any sequence $(y_i)_{i\geq 1}$ converging to a point in $Y_1 \setminus Y_0$, the sequence 
$(\int_{X_{y_i}} e^{-2\psi} \omega_X ^m /f^*\omega_Y ^n )_{i\geq 1}$ will not tend to $0$. 
\end{remark}

\begin{proof}
Fix an open set $\Delta_1$ such that $\Delta \Subset \Delta_1\Subset Y_1 \cap U_i$.
Let $y_0$ be a point in $\Delta \cap Y_0$ and let $c_{y_0}$ be a constant such that 
\begin{equation}\label{added1}
|c_{y_0} |^2 \int_{X_{y_0}} e^{-2\psi} \omega_X ^m / f^*\omega_Y ^n =1 . 
\end{equation}
Let $s_{\lceil F\rceil}$ be the canonical section of $\lceil F^v +F^h\rceil$.
By applying Proposition \ref{keyextension} to $( f^{-1} (\Delta_1) , K_X + L_i, h_{L_i})$ and the section $c_{y_0} \otimes s_{\lceil F \rceil} 
\in H^0 (X_{y_0} , K_X +L_i)$, 
we can find a holomorphic section $\tau \in H^0 (f^{-1} (\Delta_1) , K_X + L_i)$ such that 
$$\tau |_{X_{y_0}} =c_{y_0} \otimes s_{\lceil F \rceil}$$
and
\begin{equation}\label{l2condition}
\int_{f^{-1} (\Delta_1) } |\tau|_{h_{L_i}} ^2 \leq C_1 \int_{X_{y_0} } |\tau|_{h_{L_i}} ^2 
= C_1 | c_{y_0} |^2 \int_{X_{y_0}} e^{-2\psi} \omega_X ^m / f^*\omega_Y ^n =C_1
\end{equation} 
where $C_1$ is a constant independent of $y_0\in \Delta \cap Y_0$.

\medskip

Set $\widetilde{\tau} :=\frac{\tau}{s_{\lceil F \rceil}}$.
Then $\widetilde{\tau}$ can be extended to a meromorphic function (we still denote it by $\widetilde{\tau}$) on $f^{-1} (\Delta_1)$ 
and \eqref{l2condition} implies that
\begin{equation}\label{l2conditionmero}
\int_{f^{-1} (\Delta_1)} |\widetilde{\tau}|^2 e^{-2\psi} \leq C_1
\end{equation}
Therefore
\begin{equation}\label{l2imply}
\DIV (\widetilde{\tau} ) \geq - \lceil F^h \rceil - \lceil F^v \rceil \qquad\text{on } f^{-1} (\Delta_1) .
\end{equation}

\bigskip

We now prove that $\widetilde{\tau}$ is in fact holomorphic on $f^{-1} (\Delta_1)$.
For every point $y\in \Delta_1 \cap Y_0$, thanks to $(b)$, $F^v \cap X_y =\emptyset$. 
Together with \eqref{l2imply} and $(c)$, we have
$$\DIV (\widetilde{\tau} |_{X_y} ) \geq - \lceil F^h |_{X_y}\rceil \qquad\text{on } X_y $$
for every $y\in \Delta_1 \cap Y_0$.
Combining this with $(d)$, 
$\widetilde{\tau} |_{X_y}$ is constant for every $y\in \Delta_1 \cap Y_0$.
Therefore $\widetilde{\tau}$ comes from a meromorphic function on $\Delta_1$. Then $\widetilde{\tau}$ does not have poles along $\supp (F^h)$
and \eqref{l2imply} implies that
$$
\DIV (\widetilde{\tau} ) \geq - \lceil F^v \rceil.
$$
Together with \eqref{importantimply2}, we can find a holomorphic function $\zeta$ on $\Delta_1$ such that $\widetilde{\tau} =\zeta \circ f$.

\bigskip

We now prove the lemma. Let $M \in \mathbb{N}$ large enough such that the $\mathbb{Q}$-divisor $\frac{1}{M-1} F^v +\frac{1}{M-1} F^h$ is klt. Thanks to \eqref{l2conditionmero} and the H\"older inequality, we have
\begin{equation}\label{holdin}
\int_{f^{-1} (\Delta_1)} |\widetilde{\tau}|^\frac{2}{M} \leq 
(\int_{f^{-1} (\Delta_1)} |\widetilde{\tau}|^2 e^{-2\psi})^{\frac{1}{M}} 
(\int_{f^{-1} (\Delta_1)} \frac{|s_{B}|^{\frac{2}{M-1}}}{|s_{F^v} s_{F^h}|^{\frac{2}{M-1}}})^{\frac{M-1}{M}} \leq C_2
\end{equation}
for some uniform constant $C_2$.
Since $\widetilde{\tau} =\zeta \circ f$ and $\zeta$ is holomorphic on $\Delta_1$ and $\Delta\Subset \Delta_1$, by applying maximal principal to $\zeta$, 
\eqref{holdin} implies that 
$$\sup_{z\in \Delta} |\zeta| (z) \leq C_3 \cdot (C_2)^M$$ 
where $C_3$
is a constant depending only on $\Delta$ and $\Delta_1$. In particular, the norm of $c_{y_0} = \tau |_{X_{y_0}} = \zeta (y_0)$ is less than $C_3 \cdot (C_2)^M$.
Combining this with \eqref{added1} and the fact that $C_2$ and $C_3$ are independent of the choice of $y_0 \in \Delta$, the lemma is proved.
\end{proof}

Now we prove the claim in the proof of Theorem \ref{theoremdirectimage}.

\begin{proof}[\bf{Proof of the claim}]

Let $U_{i,0}$ be the open set defined in Step 3 of the proof of Theorem \ref{theoremdirectimage}. 
Thanks to Theorem \ref{maintool}, $\varphi_i$ 
can be extended as a quasi-psh function on $f^{-1} (U_i)$ and satisfying
\begin{equation}\label{posicurrent}
\theta +dd^c \varphi_i \geq 0 \qquad\text{on } f^{-1} (U_i). 
\end{equation}
Let $s_{\lceil F \rceil}$ be the canonical section of $\lceil F^v + F^h \rceil$.
Then $\frac{e^{\varphi_i}}{s_{\lceil F \rceil}}$ is well defined on 
$f^{-1} (U_{i,0})\setminus (F^v + F^h)$.

\medskip

We next prove that $\frac{e^{\varphi_i}}{s_{\lceil F \rceil}}$ is uniformly upper bounded near the generic point of 
$\DIV (F^v + F^h)$.
Let $y$ be a generic point in $U_{i,0}$. By the construction of $s_y$ and \eqref{importantimply},
$\frac{s_y}{s_{\lceil F \rceil}}$ is a constant on $X_y$. 
Then $\frac{e^{\varphi_i}}{s_{\lceil F \rceil}} |_{X_y} = \frac{|s_y|_{h_i}}{s_{\lceil F \rceil}}$ is uniformly bounded on $X_y$.
Therefore $\frac{e^{\varphi_i}}{s_{\lceil F \rceil}}$ is uniformly bounded near the generic point of $\DIV (F^h)$.

For any $\Delta\Subset Y_1 \cap U_i $, thanks to Lemma \ref{uniformbound},
there exists a constant $c >0$, such that 
$$\int_{X_y} e^{-2\psi} (\omega_X ^m /f^*(\omega_Y) ^n ) \geq c \qquad\text{for every }y \in\Delta \cap Y_0 .$$
Together with the facts that
$$\int_{X_y} |\frac{s_y}{s_{\lceil F \rceil}}|^2 e^{-2\psi} =\int_{X_y} |s_y|^2 _{h_{L_i}} =1$$
and $\frac{s_y}{s_{\lceil F \rceil}}$ is constant on $X_y$, we see that
$\frac{e^{\varphi_i}}{s_{\lceil F \rceil}}$ is uniformly upper bounded on $f^{-1} (\Delta\cap Y_0 )$.
Since $\codim_Y (Y\setminus Y_1) \geq 2$ and $f_* (F_1 ^v)$ is of codimension $1$ by assumption, the function 
$\frac{e^{\varphi_i}}{s_{\lceil F \rceil}}$ is uniformly upper bounded near the generic point of $\DIV (F_1 ^v)$.

\medskip

Now we can prove the claim. 
Since $\frac{e^{\varphi_i}}{s_{\lceil F \rceil}}$ 
is proved to be uniformly upper bounded near the generic point of $\DIV ( F_1 ^v +F^h )$, 
the Lelong numbers of 
$dd^c \varphi_i$ at the generic points of $\DIV ( F_1 ^v +F^h)$ is not less than the Lelong numbers of the current 
$\lceil  F_1 ^v +F^h \rceil$ 
at the generic points of $\DIV ( F_1 ^v +F^h)$.
Together with \eqref{posicurrent}, we have
\begin{equation}
\theta +dd^c \varphi_i \geq \lceil  F_1 ^v +F^h \rceil .\qquad\text{on }f^{-1} (U_i) ,
\end{equation}
and the claim is proved.
\end{proof}

\section{Weak subadjunction}

\begin{definition} \label{definitionmodified} \cite[Defn.2.2]{Bou04} 
Let $X$ be a compact K\"ahler manifold, and let $\alpha$ be a cohomology class on $X$.
We say that $\alpha$ is a modified K\"ahler class if it contains
a K\"ahler current $T$ such that the generic Lelong number $\nu(T, D)$ is zero for every prime divisor $D \subset X$.
\end{definition}

By \cite[Prop.2.3]{Bou04} a cohomology class is modified K\"ahler if and only if there exists a modification
$\holom{\mu}{\tilde X}{X}$ and a K\"ahler class $\tilde \alpha$ on $\tilde X$ such that $\mu_* \tilde \alpha = \alpha$.
For our purpose we have to fix some more notation:

\begin{definition} \label{definitionlogresolution}
Let $X$ be a compact K\"ahler manifold, and let $\alpha$ be a modified K\"ahler class on $X$.
A log-resolution of $\alpha$ is a bimeromorphic morphism $\holom{\mu}{\tilde X}{X}$ from a compact K\"ahler manifold $\tilde X$ 
such that  the exceptional locus is a simple normal crossings divisor $\sum_{j=1}^k E_j$ and there exists
a K\"ahler class $\tilde \alpha$ on $\tilde X$ such that $\mu_* \tilde \alpha = \alpha$.
\end{definition}

The definition can easily be extended to arbitrary big classes by using the Boucksom's Zariski decomposition
\cite[Thm.3.12]{Bou04}.

\begin{remark} \label{remarkcoeffs}
If $\holom{\mu}{\tilde X}{X}$ is a log-resolution of $\alpha$ one can write
$$
\mu^* \alpha = \tilde \alpha + \sum_{j=1}^k r_j E_j
$$
and $r_j>0$ for all $j \in \{1, \ldots, k\}$. For $\R$-divisors this is known as the the negativity lemma 
\cite[3.6.2]{BCHM10}, in the analytic setting we proceed as follows: let $T \in \alpha$ be a current with analytic
singularities such that the generic Lelong $\nu(T,D)$ is zero for every prime divisor $D \subset X$.
Resolving the ideal sheaf defining $T$ and pulling back we obtain
$$
\mu^* \alpha = \alpha' + \sum_{j=1}^k r_j' E_j \geq \mu^* \omega
$$
where $\omega$ is a K\"ahler form, $r_j'>0$ for all $j \in \{1, \ldots, k\}$ and $\alpha'$ is semi-positive
with null locus equal to $\cup_{j=1}^k E_j$. For $0<\varepsilon_j \ll 1$ the class $\tilde \alpha :=
\alpha' - \sum_{j=1}^k \varepsilon_j E_j$ is K\"ahler, so the statement holds by setting $r_j:=r_j' + \varepsilon_j$.
\end{remark}

\begin{definition} \label{definitionmaxlc}
Let $X$ be a compact K\"ahler manifold, and let $\alpha$ be a modified K\"ahler class on $X$.
A subvariety $Z \subset X$ is a maximal lc centre if there exists a log-resolution $\holom{\mu}{\tilde X}{X}$ of $\alpha$ 
with exceptional locus $\sum_{j=1}^k E_j$ such that the following holds:
\begin{itemize}
\item $Z$ is an irreducible component of $\mu(\supp \sum_{j=1}^k E_j)$;
\item if we write
$$
K_{\tilde X} + \tilde \alpha = \mu^* (K_X+\alpha) + \sum_{j=1}^k d_j E_j,
$$
then $d_j \geq -1$ for every $E_j$ mapping onto $Z$ and (up to renumbering) we have $\mu(E_1)=Z$ and $d_1=-1$. 
\end{itemize}
\end{definition}

Following the terminology for singularities of pairs we call the coefficients $d_j$ the discrepancies of $(X, \alpha)$.
Note that this terminology is somewhat abusive since $d_j$ is not determined by the class $\alpha$ but depends on the choice
of $\tilde \alpha$ (hence implicitly on the
choice of a K\"ahler current $T$ in $\alpha$ that is used to construct the log-resolution). Similarly it would be more appropriate
to define $Z$ as an lc centre of the pair $(X, T)$ with $[T] \in \alpha$. Since most of the time we will only work with the
cohomology class we have chosen to use this more convenient terminology.

We can now prove the weak subadjunction formula:

\begin{proof}[Proof of Theorem \ref{theoremweaksubadjunction}]

{\em Step 1. Geometric setup.}
Since $Z \subset X$ is a maximal lc centre of $(X, \alpha)$ there exists a log-resolution $\holom{\mu}{\tilde X}{X}$ of $\alpha$ 
with exceptional locus $\sum_{j=1}^k E_j$ such that 
$Z$ is an irreducible component of $\mu(\supp \sum_{j=1}^k E_j)$
and
\begin{equation} \label{decompalpha}
K_{\tilde X} + \tilde \alpha = \mu^* (K_X+\alpha) + \sum_{j=1}^k d_j E_j,
\end{equation}
satisfies $d_j \geq -1$ for every $E_j$ mapping onto $Z$ and (up to renumbering) we have $\mu(E_1)=Z$ and $d_1=-1$. 
Let $\holom{\pi}{X'}{X}$ be an embedded resolution of $Z$, then (up to blowing up further $\tilde X$) we can
suppose that there exists a factorisation $\holom{\psi}{\tilde X}{X'}$. Let $Z' \subset X'$ be the strict transform of $Z$.
Since $\pi$ is an isomorphism in the generic point of $Z'$, the divisors $E_j$ mapping onto $Z'$ via $\psi$ are exactly those mapping
onto $Z$ via $\mu$. Denote by $Q_l \subset Z'$ the prime divisors that are images of divisors
$E_1 \cap E_j$ via $\psi|_{E_1}$. Then we can suppose 
(up to blowing up further $\tilde X$) that the divisor
$$
\sum_l (\psi|_{E_1})^* Q_l + \sum_{j=2}^k E_1 \cap E_j
$$
has a support with simple normal crossings. 
We set 
$$
f:=\psi|_{E_1}, \qquad \mbox{and} \qquad D= - \sum_{j=2}^k d_j D_j
$$
where $D_j := E_j \cap E_1$.
Note also that the desingularisation $\pi|_{Z'}$ factors through the 
normalisation $\holom{\nu}{\tilde Z}{Z}$, so we have a bimeromorphic morphism
$\holom{\tau}{Z'}{\tilde Z}$ such that $\pi|_{Z'}=\nu \circ \tau$.
We summarise the construction in a commutative diagram:
$$
\xymatrix{
& & E_1 \ar[lldd]_{f:=\psi|_{E_1}} \ar[rrdd]  \ar @{^{(}->}[d] & &
\\
& & \tilde X  \ar[ld]_\psi \ar[rd]^\mu  & &
\\
Z' \ar @{^{(}->}[r] \ar[rrd]_\tau & X'  \ar[rr]^\pi & & X & Z \ar @{_{(}->}[l]
\\
& & \tilde Z \ar[rru]_\nu & &
}
$$

A priori there might be more than one divisor with discrepancy $-1$ mapping onto $Z$, but we can
use the tie-breaking technique which is well-known in the context of singularities of pairs:
recall that the class $\tilde \alpha$ is K\"ahler which is an open property. Thus we can choose $0 < \varepsilon_j \ll 1$
for all $j \in \{2, \ldots, k\}$ such that the class
$\tilde \alpha +  \sum_{j=2}^k \varepsilon_j E_j$ is K\"ahler.
The decomposition
$$
K_{\tilde X} + (\tilde \alpha+  \sum_{j=2}^k \varepsilon_j E_j) = \mu^* (K_X+\alpha) - E_1 + \sum_{j=2}^k (d_j+\varepsilon_j) E_j
$$
still satisfies the properties in Definition \ref{definitionmaxlc} and $E_1$ is now the unique divisor with discrepancy $-1$ mapping onto
$Z$. Note that up to perturbing $\varepsilon_j$ we can suppose that
$d_j+\varepsilon_j$ is rational for every $j \in \{1, \ldots, k\}$. In order to simplify the notation we will suppose without loss of generality, that these properties already holds for the decomposition \eqref{decompalpha}.

{\em Outline of the strategy.} The geometric setup above is analogous to the proof of Kawamata's subadjunction formula
\cite[Thm.1]{Kaw98} and as in Kawamata's proof our aim is now to apply the positivity theorem \ref{theoremdirectimage} to $f$ to relate $K_{Z'}$ and $(\pi|_{Z'})^*(K_X+\alpha)|_Z$. However since we deal with an lc centre that is not minimal we encounter
some additional problems: the pair $(E_1, D)$ is not necessarily (sub-)klt and the centre $Z$ might not be regular in codimension one.
In the end this will not change the relation between $K_{Z'}$ and $(\pi|_{Z'})^*(K_X+\alpha)|_Z$, but it leads to some technical computations which will be carried out in the Steps 3 and 4.

{\em Step 2. Relative vanishing.}
Note that the $\Q$-divisor $-K_{\tilde X} - E_1 +  \sum_{j=2}^k d_j E_j$ is $\mu$-ample since
its class is equal to $\tilde \alpha$ on the $\mu$-fibres. Thus we can apply the relative Kawamata-Viehweg
theorem (in its analytic version \cite[Thm.2.3]{Anc87} \cite{Nak87}) to obtain that
$$
R^1 \mu_* \sO_{\tilde X}(-E_1 + \sum_{j=2}^k \lceil d_j \rceil E_j) = 0.
$$
Pushing the exact sequence
$$
0 \rightarrow \sO_{\tilde X}(-E_1 + \sum_{j=2}^k \lceil d_j \rceil E_j) 
\rightarrow \sO_{\tilde X}(\sum_{j=2}^k \lceil d_j \rceil E_j)
\rightarrow \sO_{E_1}(\lceil -D \rceil)
\rightarrow 0
$$
down to $X$, the vanishing of $R^1$ yields a surjective map
\begin{equation} \label{surjection1}
\mu_* (\sO_{\tilde X}(\sum_{j=2}^k \lceil d_j \rceil E_j)) \rightarrow (\mu|_{E_1})_*(\sO_{E_1}(\lceil -D \rceil)).
\end{equation}
Since all the divisors $E_j$ are $\mu$-exceptional, we see that $\mu_* (\sO_{\tilde X}(\sum_{j=2}^k \lceil d_j \rceil E_j))$
is an ideal sheaf $\sI$. Moreover, since $d_j>-1$ for all $E_j$ mapping onto $Z$ 
the sheaf $\sI$ is isomorphic to the structure sheaf in the generic point of $Z$ . 
In particular $(\mu|_{E_1})_*(\sO_{E_1}(\lceil -D \rceil ))$
has rank one.

{\em Step 3. Application of the positivity result.}
By the adjunction formula we have
\begin{equation} \label{help0}
K_{E_1} + \tilde \alpha|_{E_1} - \sum_{j=2}^k d_j (E_j \cap E_1) 
= f^* (\pi|_{Z'})^*(K_X+\alpha)|_Z.
\end{equation}
Since $f$ coincides with $\mu|_{E_1}$ over the generic point of $Z'$, we know by Step 2 that 
the direct image sheaf  $f_* (\sO_{E_1}( \lceil -D \rceil))$ has rank one.
In particular $f$ has connected fibres.

In general the boundary $D$ does not satisfy the conditions a) and b) in Theorem \ref{theoremdirectimage},
however we can still obtain some important information by applying Theorem \ref{theoremdirectimage} for
a slightly modified boundary:
note first that the fibration $f$ is equidimensional over the complement of a codimension two set. In particular
the direct image sheaf $f_* (\sO_{E_1}( \lceil -D \rceil))$ is reflexive \cite[Cor.1.7]{Ha80}, 
hence locally free, on the complement of a codimension two set. Thus we can consider the 
first Chern class $c_1(f_* (\sO_{E_1}(\lceil -D \rceil)))$ (cf. Definition \ref{definitionc1}). Set 
$$
L := (\pi|_{Z'})^*(K_X+\alpha)|_Z - K_{Z'},
$$
then we claim that 
\begin{equation} \label{withboundary}
\left(
L+c_1(f_* (\sO_{E_1}(\lceil -D \rceil)))
\right) 
\cdot \omega_1' \cdot \ldots \cdot \omega'_{\dim Z-1} \geq 0
\end{equation} 
for any collection of nef classes $\omega_j'$ on $Z'$.

{\em Proof of the inequality \eqref{withboundary}.}
In the complement of a codimension two subset $B \subset Z'$ the 
fibration $f|_{\fibre{f}{Z' \setminus B}}$ is equidimensional, so the
direct image sheaf $\sO_{E_1}(\lceil -D^v \rceil)$
is reflexive. Since it has rank one we thus can write 
$$
f_*(\sO_{E_1}(\lceil -D^v \rceil)) \otimes \sO_{Z' \setminus B}
= \sO_{Z' \setminus B}(\sum e_l Q_l)
$$ 
where $e_l \in \Z$ and $Q_l \subset Z'$ are the prime divisors introduced in the geometric setup. 
If $e_l>0$ then $e_l$ is the largest integer such that
$$
(f|_{\fibre{f}{Z' \setminus B}})^*(e_l  Q_l) \subset \lceil -D^v \rceil.
$$
In particular if $D_j$ maps onto $Q_l$, then $d_j>-1$. If $e_l<0$ there exists a divisor $D_j$
that maps onto $Q_l$ such that $d_j \leq -1$. Moreover if $w_j$ is the coefficient of $D_j$ in the pull-back 
$(f|_{\fibre{f}{Z' \setminus B}})^* Q_l$, then $e_l$ is the largest integer such that 
$d_j - e_l w_j >-1$ for every divisor $D_j$ mapping onto $Q_l$. 
Thus if we set
$$
\tilde D := D + \sum e_l f^*  Q_l,
$$
then $\tilde D$ has normal crossings support (cf. Step 1) and satisfies the condition a) in Theorem \ref{theoremdirectimage}. 
Moreover if we denote by $\tilde D= \tilde D^h + \tilde D^v$ the decomposition
in horizontal and vertical part, then $\tilde D^h=D^h$ and $\tilde D^v= D^v + \sum e_l f^*  Q_l$.
Since we did not change the horizontal part, the direct image $f_* (\sO_{E_1}( \lceil -\tilde D \rceil))$ has rank one.
Since $\sum e_l f^*  Q_l$ has integral coefficients, the projection formula shows that
$$
(f_* (\sO_{E_1}( \lceil -\tilde D^v \rceil)))^{**} \simeq (f_* (\sO_{E_1}( \lceil -D^v \rceil)))^{**} 
\otimes \sO_{Z'}(- \sum e_l Q_l)
 \simeq \sO_{Z'}.
$$
Thus we satisfy the condition b) in Theorem \ref{theoremdirectimage}. Finally note that
$$
K_{E_1/Z} + \tilde \alpha|_{E_1}+ \tilde D = f^* (L+\sum e_l Q_l).
$$
So if we set $\tilde L:= L+\sum e_l Q_l$, then
\begin{equation} \label{invariance}
\tilde L + c_1(f_* (\sO_{E_1}( \lceil -\tilde D \rceil))) = L + c_1(f_* (\sO_{E_1}( \lceil -D \rceil))). 
\end{equation}
Now we apply Theorem \ref{theoremdirectimage} and obtain
$$
\tilde L \cdot \omega_1' \cdot \ldots \cdot \omega_{\dim Z'-1}' \geq 0.
$$
Yet by the conditions a) and b) there exists an ideal sheaf $\sI$ on $Z'$ that has cosupport of codimension at least two 
and $f_*(\sO_{E_1}( \lceil -\tilde D \rceil)) \simeq \sI \otimes \sO_{Z'}(B)$ with $B$ an effective divisor on $Z'$.
Thus $c_1(f_*(\sO_{E_1}( \lceil -\tilde D \rceil)))$ is represented by the effective divisor $B$
and \eqref{withboundary} follows from \eqref{invariance}.

{\em Step 4. Final computation.}
In view of our definition of the intersection product
on $\tilde Z$ (cf. Definition \ref{definitionintersection}) we are done if we prove that
$$
L \cdot \tau^* \omega_1 \cdot \ldots \cdot \tau^* \omega_{\dim Z-1} \geq 0
$$
where the $\omega_j$ are the nef cohomology classes from the statement of Theorem \ref{theoremweaksubadjunction}.
We claim that
\begin{equation} \label{bidual}
c_1(f_* (\sO_{E_1}(\lceil -D \rceil))) = - \Delta_1 + \Delta_2
\end{equation}
where $\Delta_1$ is an effective divisor and $\Delta_2$ is a divisor such that $\pi|_{Z'}(\supp \Delta_2)$ has codimension
at least two in $Z$. Assuming this claim for the time being let us see how to conclude:
by \eqref{withboundary} we have
\begin{equation} \label{help1}
(L+c_1(f_* (\sO_{E_1}(\lceil -D \rceil)))) \cdot \tau^* \omega_1 \cdot \ldots \cdot \tau^* \omega_{\dim Z-1} \geq 0.
\end{equation}
Since the normalisation $\nu$ is finite and 
$\pi|_{Z'}(\supp \Delta_2)$ has codimension at least two in $Z$, we see that $\tau(\supp \Delta_2)$ has codimension at least two 
in $\tilde Z$. Thus we have 
$$
c_1(f_* (\sO_{E_1}(\lceil -D \rceil))) \cdot \tau^* \omega_1 \cdot \ldots \cdot \tau^* \omega_{\dim Z-1}
= - \Delta_1 \cdot \tau^* \omega_1 \cdot \ldots \cdot \tau^* \omega_{\dim Z-1} \leq 0.
$$
Hence the statement follows from \eqref{help1}.

{\em Proof of the equality \eqref{bidual}.}
Applying as in Step 2 the relative Kawamata-Viehweg vanishing theorem to the morphism $\psi$ we obtain a surjection
$$
\psi_* (\sO_{\tilde X}(\sum_{j=2}^k \lceil d_j \rceil E_j)) \rightarrow (\psi|_{E_1})_* (\sO_{E_1}(\lceil -D \rceil))
$$
In order to verify \eqref{bidual} note first that some of the divisors $E_j$ might not be $\psi$-exceptional, so it is not clear
if $\psi_* (\sO_{\tilde X}(\sum_{j=2}^k \lceil d_j \rceil E_j))$ is an ideal sheaf. However if we restrict the surjection \eqref{surjection1}
to $Z$ we obtain a surjective map
\begin{equation} \label{surjection2}
\sI \otimes_{\sO_X} \sO_Z \rightarrow (\pi|_{Z'})_* (f_* (\sO_{E_1}(\lceil -D \rceil))),
\end{equation}
where $\sI$ is the ideal sheaf introduced in Step 2.
There exists an analytic set $B \subset Z$ of codimension at least two such that 
$$
Z' \setminus \fibre{\pi}{B} \rightarrow Z \setminus B
$$
is isomorphic to the normalisation of $Z \setminus B$. In particular the restriction of $\pi$ to $Z' \setminus \fibre{\pi}{B}$
is finite, so the natural map
$$
(\pi|_{Z'})^* (\pi|_{Z'})_* (f_* (\sO_{E_1}(\lceil -D \rceil)))
\rightarrow f_* (\sO_{E_1}(\lceil -D \rceil))
$$
is surjective on $Z' \setminus \fibre{\pi}{B}$. Pulling back is right exact, so composing with the surjective map \eqref{surjection2}
we obtain a map from an ideal sheaf to $f_* (\sO_{E_1}(\lceil -D \rceil))$ that is surjective 
on $Z' \setminus \fibre{\pi}{B}$. An ideal sheaf is torsion-free, so this map is an isomorphism 
onto its image in $\sJ \subset f_* (\sO_{E_1}(\lceil -D \rceil))$. In the complement of a codimension two set the sheaf
$\sJ$ corresponds to an antieffective divisor $-\Delta_1'$. Since the inclusion
$\sJ \subset f_* (\sO_{E_1}(\lceil -D \rceil))$ is an isomorphism on $Z' \setminus \fibre{\pi}{B}$, there
exists an effective divisor $\Delta_2'$ with support in $\fibre{\pi}{B}$ such that $c_1(f_* (\sO_{E_1}(\lceil -D \rceil)))=-\Delta_1'+\Delta_2'$. We denote by $\Delta_1$ the part of $\Delta_1'$
whose support is not mapped into $B$ (hence maps into the non-normal locus of $Z \setminus B$) 
and set $\Delta_2 := \Delta_2' + \Delta_1 - \Delta_1'$.
Then we have $c_1(f_* (\sO_{E_1}(\lceil -D \rceil))) = - \Delta_1 + \Delta_2$
and the support of $\Delta_2$ maps into $B$.
Since $B$ has codimension at least two this proves the equality \eqref{bidual}.
\end{proof}

\begin{remark}
In Step 3 of the proof of Theorem \ref{theoremweaksubadjunction} above we introduce 
a ``boundary'' $c_1(f_*(\sO_M( \lceil -D \rceil)))$ so that we can apply Theorem \ref{theoremdirectimage}.
One should note that this divisor is fundamentally different from the divisor $\Delta$ appearing
in \cite[Thm.1, Thm.2]{Kaw98}.
In fact for a minimal lc centre Kawamata's arguments show that $c_1(f_*(\sO_M( \lceil -D \rceil)))=0$,
his boundary divisor $\Delta$ is defined in order to obtain the stronger result 
that $L-\Delta$ is nef. We have to introduce $c_1(f_*(\sO_M( \lceil -D \rceil)))$ since we want to deal with non-minimal
centres.
\end{remark}

\section{Positivity of relative adjoint classes, part 2}

{\bf Convention :} In this section, we use the following convention. Let $U$ be a open set and $(f_m)_{m\in \N}$ be a sequence of smooth functions on $U$. We say that
$$\|f_m\|_{C^{\infty} (U)} \rightarrow 0 ,$$
if for every open subset $V\Subset U$ and every index $\alpha$, we have
$$ \|\partial^\alpha f_m\|_{C^0 (V)} \rightarrow 0 .$$
Similarly, in the case $(f_m)_{m\in \N}$ are smooth formes, we say that $\|f_m\|_{C^{\infty} (U)} \rightarrow 0$ if every component tends to $0$ in the above sense.

\medskip

Before giving the main theorem of this section, we need two preparatory lemmas. The first comes from \cite[Part II, Thm 1.3]{Lae02} :

\begin{lemma}\label{approxlemma}\cite[Part II, Thm 1.3]{Lae02}
Let $X$ be a compact K\"{a}hler manifold and let $\alpha$ be a closed smooth real $2$-form on $X$. 
Then we can find a strictly increasing sequence of integers $(s_m)_{m\geq 1}$ and a sequence of hermitian line bundles (not necessary holomorphic) 
$(F_m, D_{F_m} , h_{F_m})_{m \geq 1}$ on $X$ such that 
\begin{equation}\label{importlimito}
\lim_{m\rightarrow +\infty}\| \frac{\sqrt{-1}}{2\pi}\Theta_{h_{F_m}} (F_m) - s_m  \alpha\|_{C^{\infty} (X)} =0.
\end{equation}
Here $D_{F_m}$ is a hermitian connection with respect to the smooth hermitian metric $h_{F_m}$ and $\Theta_{h_{F_m}} (F_m) = D_{F_m} \circ D_{F_m}$.

Moreover, let $(W_j)$ be a small Stein cover of $X$ and let $e_{F_m,j}$ be a basis of an isometric trivialisation of $F_m$ over $W_j$ i.e., 
$\| e_{F_m,j}\|_{h_m}= 1$.
Then we can ask the hermitian connections $D_{F_m}$ (under the basis $e_{F_m, j}$) to satisfy the following additional condition:
for the $(0,1)$-part of $D_{F_m}$  on $W_j$ : $D_{F_m} '' =\overline{\partial} + \beta^{0,1} _{m,j}$, we have
\begin{equation}\label{importlimito2}
\|\frac{1}{s_m } \beta^{0,1} _{m,j} \|_{C^{\infty} (W_j)} \leq C \|\alpha\|_{C^{\infty} (X)} ,
\end{equation}
where $C$ is a uniform constant independent of $j$ and $m$.
\end{lemma}

\begin{proof}
Thanks to \cite[Part II, Thm 1.3]{Lae02}, we can find a strictly increasing integer sequence $(s_m)_{m\geq 1}$ 
and closed smooth $2$-forms $(\alpha_m)_{m\geq 1}$ on $X$, such that
$$\lim_{m\rightarrow +\infty}\|\alpha_m - s_m \alpha\|_{C^{\infty} (X)} =0 \qquad\text{and} \qquad \alpha_m \in H^2 (X, \Z) .$$
Since $(W_j)$ are small Stein open sets, we can find some smooth $1$-forms $\beta_{m, j}$ on $W_j$ such that
\begin{equation}\label{lemmaadd}
\frac{1}{2\pi} \cdot d\beta_{m,j} = \alpha_m \text{ on }W_j \qquad\text{and}\qquad \|\frac{1}{s_m} \beta_{m,j} \|_{C^{\infty} (W_j)} \leq C\|\alpha\|_{C^{\infty} (X)}
\end{equation}
for a constant $C$ independent of $m$ and $j$.

By using the standard construction (cf. for example \cite[V, Thm 9.5]{Dem}), the form $(\beta_{m,j})_j$ induces a hermitian line bundle $(F_m, D_m, h_{F_m})$ on $X$
such that $D_m = d + \frac{\sqrt{-1}}{2\pi} \beta_{m,j}$ with respect to an isometric trivialisation over $W_j$. 
Then 
$$\|\frac{\sqrt{-1}}{2\pi}\Theta_{h_{F_m}} (F_m) - s_m \alpha\|_{C^{\infty} (X)}= \|\alpha_m - s_m\alpha\|_{C^{\infty} (X)} \rightarrow 0 .$$
Let $\beta^{0,1} _{m,j}$ be the $(0,1)$-part of $\beta_{m,j}$. Then
\eqref{lemmaadd} implies \eqref{importlimito2}.
\end{proof}

Now we can prove the main theorem of this section.

\begin{theorem} \label{theoremdirectimage2}
Let $X$ and $Y$ be two compact K\"ahler manifolds and let $f: X\rightarrow Y$ be a surjective map with connected fibres such that the general fibre $F$ is simply connected and
$$
H^0(F, \Omega^2_F) = 0. 
$$
Let $\omega$ be a K\"ahler form on $X$ such that $c_1 (K_F) +[\omega|_F]$ is a pseudoeffective class.
Then $c_1 (K_{X/Y})+[\omega]$ is pseudoeffective.
\end{theorem}

\begin{proof}
Being pseudoeffective is a closed property, so
we can assume without loss of generality 
that $c_1 (K_F) + [\omega |_F]$ is big on $F$.

{\em Step 1: Preparation, Stein Cover.}

Fix two K\"{a}hler metrics $\omega_X$, $\omega_Y$ on $X$ and $Y$ respectively.
Let $h$ be the smooth hermitian metric on $K_{X/Y}$ induced by $\omega_X$ and $\omega_Y$.
Set $\alpha := \frac{\sqrt{-1}}{2\pi}\Theta_h (K_{X/Y})$.
Thanks to Lemma \ref{approxlemma}, there exist a strictly increasing sequence 
of integers $(s_m)_{m\geq 1}$ 
and a sequence of hermitian line bundles (not necessary holomorphic)
$(F_m, D_{F_m} , h_{F_m})_{m \geq 1}$ on $X$ such that 
\begin{equation}\label{importlimitor}
\| \frac{\sqrt{-1}}{2\pi} \Theta_{h_{F_m}} (F_m) - s_m  (\alpha+\omega)\|_{C^{\infty} (X)}\rightarrow 0 .
\end{equation}

\smallskip

By our assumption on $F$ we can find a non empty Zariski open subset $Y_0$ of $Y$ such that 
$f$ is smooth over $Y_0$ and $R^i f_* \mathcal{O}_X = 0$ on $Y_0$ for every $i=1,2$.
Let $(U_i)_{i \in I}$ be a Stein cover of $Y_0$.
Therefore
\begin{equation}\label{importantprop}
H^{0,2} (f^{-1} (U_i), \R) = 0  \qquad\text{for every } i \in I. 
\end{equation}

\medskip

{\em Step 2: Construction of the approximate holomorphic line bundles.}

Let $\Theta_{h_{F_m}} ^{(0,2)} (F_m) $ be the $(0,2)$-part of $\Theta_{h_{F_m}} (F_m)$.
Thanks to \eqref{importantprop} and \eqref{importlimitor}, $\Theta_{h_{F_m}} ^{(0,2)} (F_m) $ is $\overline{\partial}$-exact on $f^{-1}(U_i)$
and 
\begin{equation}\label{importlimitordbar}
\| \Theta_{h_{F_m}} ^{(0,2)} (F_m) \|_{C^{\infty} (f^{-1}(U_i))}\rightarrow 0 .
\end{equation}

We first construct a sequence of $(0,1)$-formes $\beta_m$ on $f^{-1}(U_i)$ such that 
\begin{equation}\label{importlimitordbarsol}
\Theta_{h_{F_m}} ^{(0,2)} (F_m) =\overline{\partial} \beta_m \qquad\text{and} \qquad
\| \beta_m \|_{C^{\infty} (f^{-1}(U_i))}\rightarrow 0 .
\end{equation}
In fact, for every $y\in U_i$, as $X_y$ is compact and $H^{0,2} (X_y)=0$, we can find smooth $(0,1)$-forms $\theta_m$ on $f^{-1}(U_i)$
such that for every $y\in U_i$
\begin{equation}\label{importlimitordbarsol1}
(\Theta_{h_{F_m}} ^{(0,2)} (F_m) -\overline{\partial} \theta_m ) |_{X_y} =0 \qquad\text{and} \qquad
\| \theta_m \|_{C^{\infty} (f^{-1}(U_i))}\rightarrow 0 .
\end{equation}
Therefore $\Theta_{h_{F_m}} ^{(0,2)} (F_m) -\overline{\partial} \theta_m = \sum_j f^\star (d\overline{t}_j) \wedge \gamma_{m,j}$,
where $(d\overline{t}_j)$ is a basis of $\wedge^{0,1} (U_i)$ and $\|\gamma_{m,j}\|_{C^{\infty} (f^{-1}(U_i))}\rightarrow 0$. 
Note that $\Theta_{h_{F_m}} ^{(0,2)} (F_m) -\overline{\partial} \theta_m$ is $\overline{\partial}$-closed. 
Then $\overline{\partial} \gamma_{m,j} |_{X_y} =0$.
As $H^{0,1} (X_y)=0$, we can find $\theta' _{m,j}$ on $f^{-1} (U_j)$ such that $(\gamma_{m,j} - \overline{\partial}\theta' _{m,j} ) |_{X_y}=0$
and $\| \theta' _{m,j} \|_{C^{\infty} (f^{-1}(U_i))}\rightarrow 0$.
As a consequence, 
$$\Theta_{h_{F_m}} ^{(0,2)} (F_m) -\overline{\partial} (\theta_m + \sum_j f^\star (d \overline{t}_j) \wedge \theta' _{m,j}) =f^\star \gamma$$
for some closed $(0,2)$-form $\gamma$ on $U_i$ and 
$ \|\gamma\|_{C^{\infty} (U_i)}
\rightarrow 0$.
Together with the fact that $U_i$ is Stein, we can thus find $\beta_m$ satisfies \eqref{importlimitordbarsol}.

\medskip

Thanks to \eqref{importlimitordbarsol}, we can find holomorphic line bundles $L_{i,m}$ on $f^{-1} (U_i)$ 
equipped with smooth hermitian metrics $h_{i,m}$ such that 
\begin{equation}\label{addednew}
\| \frac{\sqrt{-1}}{2\pi} \Theta_{h_{F_m}} (F_m) - \frac{\sqrt{-1}}{2\pi} \Theta_{h_{i,m}} (L_{i,m})  \|_{C^{\infty} (f^{-1} (U_i))}\rightarrow 0 . 
\end{equation}
By construction, we have
$$\frac{\sqrt{-1}}{2\pi} \Theta_{h_{i,m}} ( L_{i,m}) - s_m \frac{\sqrt{-1}}{2\pi}  \Theta_{h} (K_{X/Y}) =\frac{\sqrt{-1}}{2\pi} \Theta_{h_{i,m}} ( L_{i,m}) -s_m \alpha
$$
$$= (\frac{\sqrt{-1}}{2\pi} \Theta_{h_{i,m}} (L_{i,m}) - \frac{\sqrt{-1}}{2\pi} \Theta_{h_{F_m}} (F_m)) + (\frac{\sqrt{-1}}{2\pi} \Theta_{h_{F_m}} (F_m) -s_m (\alpha +\omega)) + 
s_m \omega  .$$
Thanks to the estimates \eqref{importlimitor} and \eqref{addednew}, the first two terms of the right-hand side of the above equality tends to $0$.
Therefore we can find a sequence of open sets $U_{i,m} \Subset U_i$,
such that $\cup_{m\geq 1} U_{i,m} =U_i$, $U_{i,m} \Subset U_{i, m+1}$ for every $m\in \bN$,
and 
\begin{equation}\label{positivelinebundle}
\frac{\sqrt{-1}}{2\pi} \Theta_{h_{i,m}} ( L_{i,m}) - s_m \frac{\sqrt{-1}}{2\pi}  \Theta_{h} (K_{X/Y}) \geq 0\qquad\text{on }f^{-1} (U_{i,m}) .
\end{equation}

\medskip

{\em Step 3: Construction of Bergman kernel type metrics.}

Let $\varphi_{i,m}$ be the $s_m$-Bergman kernel associated to the pair (cf. Remark \ref{bergmankernelconstr})
\begin{equation}\label{pair}
(L_{i,m} = s_m K_{X/Y} + (L_{i,m} - s_m K_{X/Y}) , h_{i,m}) 
\end{equation}
i.e., $\varphi_{i,m} (x) := \sup\limits_{g\in A} \frac{1}{s_m} \ln |g|_{ h_{i,m}} (x)$, where 
\begin{equation}\label{unitnormcond}
A:=\{ g \ | \ g \in H^0 (X_{f(x)} , L_{i,m}) , \int_{X_{f (x)}} |g |_{h_{i,m}} ^{\frac{2}{s_m}} \omega_X ^{\dim X} /f^*\omega_Y ^{\dim Y} =1\}.
\end{equation}
Thanks to \eqref{positivelinebundle}, we can apply Theorem \ref{maintool} to the pair \eqref{pair} over $f^{-1} (U_{i,m})$.
In particular, we have
\begin{equation}\label{bergman}
(\alpha +\omega) + dd^c \varphi_{i,m} \geq 0 \qquad\text{on }f^{-1} (U_{i,m}) . 
\end{equation}

\medskip

We recall that $\varphi_{i,m}$ is invariant after a normalisation of $h_{i,m}$, namely, 
if we replace the metric $h_{i,m} |_{X_y}$ by $c\cdot h_{i,m} |_{X_y}$ for some constant $c >0$,
the associated Bergman kernel function $\varphi_{i,m} |_{X_y}$ is unchanged cf. Remark \ref{bergmankernelconstr} (3).

\medskip

Let $y\in  U_i$ be a generic point. Thanks to the above remark and \eqref{addednew},
we can find a constant $c_y >0$ independent of $m$, such that $c_y \leq h_{i,m} |_{X_y} \leq c_y ^{-1}$.
Therefore, by mean value inequality, $\varphi_{i,m} |_{X_y}$ is uniformly upper bounded.
Therefore we can define 
$$\varphi_i := \lim_{k\rightarrow +\infty} (\sup\limits_{m\geq k} \varphi_{i,m})^\star ,$$
where $\star$ is the u.s.c regularization. 
Thanks to \eqref{unitnormcond}, $\varphi_i$ cannot be identically $-\infty$.
Therefore $\varphi_{i}$ is a quasi-psh. As $\cup_{m\geq 1} U_{i,m} =U_i$, \eqref{bergman} implies
\begin{equation}\label{step2aim}
\alpha +\omega + dd^c \varphi_i \geq 0 \qquad\text{on }f^{-1} (U_i)  \text{ in the sense of currents.}
\end{equation}

\medskip

{\em Step 4: Final conclusion.}

We claim that

{\bf Claim 1.} $\varphi_i =\varphi_j$ on $f^{-1} (U_i \cap U_j)$ for every $i , j$.

{\bf Claim 2.} For every small Stein open set $V$ in $X$, we can find a constant $C_V$ depending only on $V$
such that 
$$\varphi_i (x)\leq C_V \qquad\text{for every } i \text{ and } x\in V \cap f^{-1} (U_i) .$$
We postpone the proof of these two claims and finish first the proof of the theorem.

Thanks to Claim 1, $(\varphi_i)_{i\in I}$ defines a global quasi-psh function $\varphi$ on $f^{-1} (Y_0)$ and \eqref{step2aim} implies that
$$\alpha+\omega +dd^c \varphi \geq 0 \qquad\text{on }f^{-1} (Y_0).$$
Thanks to Claim 2, we have $\varphi \leq C_V$ on $V \cap f^{-1} (Y_0)$. Therefore $\varphi$ can be extended as a quasi-psh function on $V$. 
Since Claim 2 is true for every small Stein open set $V$, $\varphi$ can be extended as a quasi-psh function on $X$ and satisfies
$$\alpha+\omega +dd^c \varphi \geq 0 \qquad\text{on }X.$$
As a consequence, $c_1 (K_{X/Y}) +[\omega]$ is pseudoeffective and the theorem is proved.
\end{proof}

We are left to prove the two claims in the proof of the theorem.

\begin{lemma}
The claim 1 holds, i.e., $\varphi_i =\varphi_j$ on $f^{-1} (U_i \cap U_j)$ for every $i , j$.
\end{lemma}

\begin{proof}

Let $y\in U_i \cap U_j$ be a generic point. Thanks to \eqref{addednew}, we have
\begin{equation}\label{limitcomportement}
\lim_{m\rightarrow +\infty}\|\frac{\sqrt{-1}}{2\pi} \Theta_{h_{i,m}} ( L_{i,m}) |_{X_y} - 
\frac{\sqrt{-1}}{2\pi} \Theta_{h_{j,m}} ( L_{j,m}) |_{X_y}\|_{C^{\infty} (X_y)} = 0 . 
\end{equation}
When $m$ is large enough, \eqref{limitcomportement} implies that 
$$
c_1 ( L_{i,m} |_{X_y}) =c_1 ( L_{j,m} |_{X_y}) \in H^{1,1} (X_y) \cap H^2(X_y, \Z).
$$
As $X_y$ is simply connected, $\Pic0 (X_y)=0$. Therefore 
\begin{equation}\label{newadded2}
L_{i,m}|_{X_y} =L_{j,m} |_{X_y} \qquad\text{for }m \gg 1. 
\end{equation}
Under the isomorphism of \eqref{newadded2}, by applying $\partial\overline{\partial}$-lemma, \eqref{limitcomportement} imply the existence of constants $c_m \in\R$ and smooth functions $\tau_m \in C^{\infty} (X_y)$
such that
$$h_{i,m}=h_{j,m}e^{c_m +\tau_m} \text{ on }X_y\qquad\text{and}\qquad \lim_{m\rightarrow +\infty}\|\tau_m \|_{C^{\infty} (X_y)} = 0 .$$
Combining with the construction of $\varphi_{i,m}$ and $\varphi_{j,m}$, we know that 
$$\|\varphi_{i,m} - \varphi_{j,m} \|_{C^{0} (X_y)} \leq \|\tau_m \|_{C^{0} (X_y)} \rightarrow 0 .$$
Therefore 
\begin{equation}\label{newadded}
\varphi_i |_{X_y} =\varphi_j |_{X_y} 
\end{equation}
As \eqref{newadded} is proved for every generic point $y\in U_i \cap U_j$, we have
$$\varphi_i =\varphi_j \qquad\text{on }f^{-1} (U_i \cap U_j) .$$ 
The lemma is proved.
\end{proof}

\medskip

It remains to prove the claim 2. Note that $(L_{i,m}, h_{i,m})$ is defined only on $f^{-1}(U_i)$, we can not directly apply Proposition \ref{keyextension}
to $(L_{i,m}, h_{i,m})$.The idea of the proof is as follows. 
Thanks to the construction of $F_m$ and $L_{i,m}$, by using $\partial\overline{\partial}$-lemma, we can prove that, after multiplying by a constant (which depends on $f(x) \in Y$), 
the difference between $h_{F_m} |_{X_{f(x)}}$ and 
$ h_{i,m} |_{X_{f(x)}}$ is uniformly controlled for $m\gg 1$ \footnote{The bigness of $m\gg 1$ depends on $f(x)$.}.
Therefore $(F_m |_{X_{f(x)}} , h_{F_m})$ is not far from $(L_{i,m} |_{X_{f(x)}} , h_{i,m})$.
Note that, using again \eqref{importlimitor}, $F_m |_V$ is not far from a holomorphic line bundle over $V$. 
Combining Proposition \ref{keyextension} with these two facts, we can finally prove the claim 2. 

\begin{lemma}\label{lemma2de}
The claim 2 holds, i.e., for every small Stein open set $V$ in $X$, we can find a constant $C_V$ depending only on $V$
such that 
$$\varphi_i (x)\leq C_V \qquad\text{for every } i \text{ and } x\in V \cap f^{-1} (U_i) .$$
\end{lemma}

\begin{proof}

{\em Step 1: Global approximation}.

Fix a small Stein cover $(W_j)_{j=1}^N$ of $X$. 
Without loss of generality, we can assume that $V \Subset W_1$.
Let $(F_m, D_{F_m} , h_{F_m})_{m \geq 1}$ be the hermitian line bundles (not necessary holomorphic) constructed in the step 1 of
the proof of Theorem \ref{theoremdirectimage2}.
Let $e_{F_m ,j}$ be a basis of a isometric trivialisation of $F_m$ over $W_j$ i.e., $\| e_{F_m ,j}\|_{h_{F_m}}= 1$.
Under this trivialisation, we suppose that the $(0,1)$-part of $D_{F_m}$ on $W_j$ is
$D_{F_m} '' =\overline{\partial} + \beta^{0,1} _{m,j}$, where $\beta^{0,1} _{m, j}$ is a smooth $(0,1)$-form on $W_j$.   
By Lemma \ref{approxlemma}, we can assume that
\begin{equation}\label{importlimit1}
\|\frac{1}{s_m} \beta^{0,1} _{m,j} \|_{C^{\infty} (W_j)} \leq C_1 \|\alpha+\omega\|_{C^{\infty} (X)}
\end{equation}
for a uniform constant $C_1$ independent of $m$ and $j$.

\medskip

{\em Step 2: Local estimation near $V$.}

Thanks to \eqref{importlimitor}, we know that $F_m$ is not far from a holomorphic line bundle. In this step, we would like to give a more precise 
description of this on $W_1$.

Since $W_1$ is a small Stein open set, thanks to \eqref{importlimitor}, we can find $\{\sigma_m ^{0,1} \}_{m \geq 1}$ on $W_1$ such that
$\overline{\partial} \sigma_m ^{0,1}  = - \Theta_{h_{F_m}} ^{(0,2)} (F_m)$ and $\lim\limits_{m\rightarrow +\infty} \|\sigma_m ^{0,1}\|_{C^{\infty} (W_1)} =0$.
Then we have
\begin{equation}\label{addfin}
(D_{F,m} '' +\sigma_m ^{0,1} )^2 =0 \text{ on } W_1, 
\end{equation}
and
$$
\| \frac{\sqrt{-1}}{2\pi} \Theta_{h_{F_m} , D_{F,m} '' +\sigma_m ^{0,1} } (F_m) - s_m  (\alpha+\omega)\|_{C^{\infty} (X)}\rightarrow 0 ,
$$
where $\Theta_{h_{F_m} , D_{F,m} '' +\sigma_m ^{0,1} } (F_m) $ is the curvature for the Chern connection on $F_m$ with respect to complex structure
$D_{F,m} '' +\sigma_m ^{0,1} $ and the metric $h_{F_m}$.

Note that $\frac{\sqrt{-1}}{2\pi} \Theta_{h_{F_m} , D_{F,m} '' + \sigma_m ^{0,1} } (F_m)$ is a closed $(1,1)$-form on $W_1$.
By $\partial\overline{\partial}$-lemma, we can find smooth functions $\{ \psi_m \}_{m \geq 1}$ on $W_1$ such that
\begin{enumerate}
\item[\rm (i)] $\frac{\sqrt{-1}}{2\pi}\Theta_{h_{F_m} e^{- \psi_m}, D_{F,m} '' +\sigma_m ^{0,1}} (F_m) = s_m ( \alpha +\omega)$ on $W_1$ for every $m\in\N$. 
\footnote{Here $\Theta_{h_{F_m} e^{- \psi_m}, D_{F,m} '' +\sigma_m ^{0,1}} (F_m) $ 
is the curvature for the Chern connection on $F_m$ with respect to complex structure
$D_{F,m} '' +\sigma_m ^{0,1} $ and the metric $h_{F_m} \cdot e^{- \psi_m}$.}
\item[\rm (ii)] $\lim\limits_{m\rightarrow +\infty} (\|\sigma_m ^{0,1}\|_{C^{\infty} (W_1)} +\|\psi_m\|_{C^{\infty} (W_1)}  ) =0$.
\end{enumerate}

Thanks to \eqref{addfin}, $\beta^{0,1} _{m,1} +\sigma_m ^{0,1}$ is $\overline{\partial}$-closed. 
Applying standard $L^2$-estimate, by restricting on some a little bit smaller open subset of $W_1$ (we still denote it by $W_1$ for simplicity),
there exists a smooth function $\eta_m$ on $W_1$ such that
\begin{equation}\label{holomequation}
\overline{\partial} \eta_m =\beta^{0,1} _{m,1} +\sigma_m ^{0,1} \qquad\text{on }W_1 
\end{equation}
and
$$
\frac{1}{s_m}\|\eta_m\|_{C^{\infty} (W_1)} \leq \frac{C_2}{s_m} \|\beta^{0,1} _{m,1} +\sigma_m ^{0,1}\|_{C^{\infty} (W_1)} 
$$
for a constant $C_2$ independent of $m$.
Combining this with \eqref{importlimit1} and $(ii)$, we get
\begin{equation}\label{addded1}
\overline{\lim}_{m\rightarrow +\infty}\frac{1}{s_m}\|\eta_m\|_{C^{\infty} (W_1)} \leq C_1\cdot C_2 .
\end{equation}
Moreover, by \eqref{holomequation}, $e^{-\eta_m}\cdot e_{F_m,1}$ is a holomorphic basis of $(W_1, F_m, D_{F_m} '' +\sigma_m ^{0,1})$.

\medskip

{\em Step 3: Final conclusion.}

Let $x\in V \cap f^{-1} (U_i)$ and set $y:= f (x)$. 

{\bf Claim.}  For $m$ large enough, there exists a $\widehat{g}  \in H^0 (X_y \cap W_1 , F_m, D''_{F_m} +\sigma_m ^{0,1})$.
such that
\begin{equation}\label{importantequ1}
\int_{X_y \cap W_1} |\widehat{g}|_{h_{F_m}} ^{\frac{2}{s_m}} \omega_X ^{\dim X} /\omega_Y ^{\dim Y} \leq 2  
\end{equation}
and
\begin{equation}\label{importantequ2}
\varphi_{i,m} (x) \leq  \frac{1}{s_m} \ln |\widehat{g}|_{h_{F_m}} (x) + 2 . 
\end{equation}
We postphone the proof of the claim later and first finish the proof of our lemma.

\medskip

As $e^{-\eta_m}\cdot e_{F_m,1}$ is a holomorphic basis of $(W_1, F_m, D_{F_m} '' +\sigma_m ^{0,1})$, we have
$$\widehat{g} = f \cdot e^{-\eta_m}\cdot e_{F_m,1}$$
for some holomorphic function $f$ on $W_1 \cap X_y$.
Thanks to \eqref{addded1}, we can find a uniform constant $C_3 >0$ independent of $m$ such that 
\begin{equation}\label{unicontr}
C_3 ^{-1}\leq |e^{-\eta_m}\cdot e_{F_m,1}|_{h_{F_m}} ^{\frac{2}{s_m}} \leq C_3 \qquad\text{on } W_1 .
\end{equation}
Together with \eqref{importantequ1}, we have
$$\int_{X_y \cap W_1} |f|^{\frac{2}{s_m}} \omega_X ^{\dim X} /\omega_Y ^{\dim Y} \leq 2C_3 .$$
By applying the Ohsawa-Takegoshi extension theorem \cite[Prop 0.2]{BP10}, we know that $|f|^{\frac{2}{s_m}}$ is uniformly controled.
Together with \eqref{unicontr}, $ \frac{1}{s_m} \ln |\widehat{g}|_{h_{F_m}} (x)$ is controled by a uniform constant $C_4$.
Combining this with \eqref{importantequ2}, the lemma is proved.
\end{proof}

It remains to prove the claim in Lemma \ref{lemma2de}.

\begin{proof}[Proof of the claim in Lemma \ref{lemma2de}]

By \eqref{importlimitor} and $\Pic0 (X_y)=0$,  when $m$ is large enough, we can find a smooth
$(0,1)$-forms $\tau_m ^{0,1}$ on $X_y$ such that 
\begin{equation}\label{isoimport}
\lim_{m\rightarrow +\infty}\|\tau_m ^{0,1}\|_{C^{\infty} (X_y)} =0 \qquad\text{and}\qquad  (F_m, D_{F_m} '' + \tau_m ^{0,1} )  |_{X_y} \simeq L_{i,m} |_{X_y}. 
\end{equation}
Let $\Theta_{h_{F_m}, \tau_m ^{0,1}} (F_m |_{X_y}) $ be the curvature calculated for the Chern connection with respect to 
$h_{F_m}$ and the complex structure $D_{F_m} '' + \tau_m ^{0,1}$ for the line bundle $F_m |_{X_y}$.
Thanks \eqref{importlimitor} and \eqref{isoimport} imply that
\begin{equation}\label{curdieff}
\lim_{m\rightarrow +\infty}\| \Theta_{h_{F_m},  \tau_m ^{0,1}} (F_m  |_{X_y} )- \Theta_{h_{i,m}} (L_{i,m}  |_{X_y} )\|_{C^{\infty} (X_y)} =0 .
\end{equation}
By using $\partial\overline{\partial}$-lemma over $X_y$, under the holomorphic isomorphism of \eqref{isoimport}, \eqref{curdieff} implies the existence of
a constant $c_{m, y}$ and a smooth function $\widetilde{\psi}_m$ on $X_y$ such that 
\begin{equation}\label{metriciso}
h_{F_m} \cdot e^{-\widetilde{\psi}_m} = h_{i,m} \cdot e^{-c_{m, y}} \qquad\text{on }X_y ,
\end{equation}
and 
\begin{equation}\label{weightestim}
\lim_{m\rightarrow +\infty}\|\widetilde{\psi}_m\|_{C^{\infty} (X_y)} =0   .
\end{equation}
Here $c_{m,y}$ is a constant on $X_y$ which depends only on $m$ and $y$.

\medskip

By the definition of $\varphi_{i,m}$, there exists a $g \in H^0 (X_y ,  L_{i,m})$ such that
\begin{equation}\label{newaddequation}
\varphi_{i, m} (x)=\frac{1}{s_m} \ln |g|_{h_{i,m}} (x) \qquad\text{and}\qquad \int_{X_y } |g|_{h_{i,m}} ^{\frac{2}{s_m}} \omega_X ^{\dim X} /\omega_Y ^{\dim Y} = 1 . 
\end{equation}
Using the holomorphic isomorphism \eqref{isoimport} and the metric estimations \eqref{weightestim} and \eqref{metriciso}, we can thus find
a $\widetilde{g} \in H^0 (X_y, F_m , D'' _{F_m} +\tau_m ^{0,1})$\footnote{It means that $\widetilde{g}$ is a holomorphic section of $F_m$ on $X_y$
with respect to the complex structure $D'' _{F_m} +\tau_m ^{0,1}$.} such that 
\begin{equation}\label{importantequ}
\int_{X_y} |\widetilde{g}|_{h_{F_m}} ^{\frac{2}{s_m}} \omega_X ^{\dim X} /\omega_Y ^{\dim Y} =1 \qquad\text{and}\qquad 
\varphi_{i,m} (x) \leq \frac{1}{s_m} \ln |\widetilde{g}|_{h_{F_m}} (x) + 1
\end{equation}
where $m$ is large enough. Here we use Remark \ref{bergmankernelconstr} (3) and the fact that $c_{m,y}$ is constant on $X_y$ (although it might be very large). 

\medskip

Now we prove the claim. 
Thanks to \eqref{isoimport} and the fact that $\tau_m ^{0,1} -\sigma_m ^{0,1}$ is $\overline{\partial}$-exact on the Stein open set $X_y \cap W_1$, 
there exists some smooth functions $\zeta_m$ on $X_y \cap W_1$, such that 
$$\overline{\partial} \zeta_m = \tau_m ^{0,1} -\sigma_m ^{0,1} \qquad\text{on }X_y \cap W_1$$
and
\begin{equation}\label{addded3}
\lim_{m\rightarrow +\infty}\frac{1}{s_m}\|\zeta_m\|_{C^{\infty} (X_y \cap W_1)} 
\leq \lim_{m\rightarrow +\infty}\frac{C_y}{s_m}\|\tau_m ^{0,1} -\sigma_m ^{0,1}\|_{C^{\infty} (X_y \cap W_1)} =0 .
\end{equation}
for a constant $C_y$ independent of $m$, but depending on $y$.

\medskip

Set $\widehat{g} := e^{\zeta_m} \cdot \widetilde{g} $. Then $\widehat{g}\in H^0 (X_y \cap W_1 , F_m, D''_{F_m} +\sigma_m ^{0,1})$.
Thanks to \eqref{addded3} and \eqref{importantequ}, when $m$ is large enough, we have
\begin{equation}
\int_{X_y \cap W_1} |\widehat{g}|_{h_{F_m}} ^{\frac{2}{s_m}} \omega_X ^{\dim X} /\omega_Y ^{\dim Y} \leq 2  
\end{equation}
and
\begin{equation}
\varphi_{i,m} (x) \leq  \frac{1}{s_m} \ln |\widehat{g}|_{h_{F_m}} (x) + 2 . 
\end{equation}
The claim is proved.
\end{proof}

\section{Proof of the main theorem}

We start with an easy, but important lemma relating null locus and lc centres.

\begin{lemma} \label{lemmalccentre}
Let $X$ be a compact K\"ahler manifold, and let $\alpha$ be a nef and big class such that the null locus $\Null{\alpha}$
has no divisorial components. Let $Z \subset X$ be an irreducible component of $\Null{\alpha}$.
Then there exists a positive real number $c$ such that $Z$ is a maximal lc centre for $(X, c \alpha)$.
\end{lemma}

\begin{remark*}
The coefficient $c$ depends on the choice of $Z$, so in general the other irreducible components of $\Null{\alpha}$
will not be lc centres for $(X, c \alpha)$.
\end{remark*}

\begin{proof}
By a theorem of Collins of Tosatti \cite[Thm.1.1]{CT13} the non-K\"ahler locus $E_{nK}(\alpha)$ coincides
with the null-locus of $\Null{\alpha}$. Moreover by \cite[Thm.3.17]{Bou04} there exists a K\"ahler current $T$
with analytic singularities in the class $\alpha$ such that the Lelong set coincides with $E_{nK}(\alpha)$.
Since the non-K\"ahler locus has no divisorial components the class $\alpha$ is a modified K\"ahler class \cite[Defn.2.2]{Bou04}.
By \cite[Prop.2.3]{Bou04} the class $\alpha$ has a log-resolution $\holom{\mu}{\tilde X}{X}$ such that $\mu_* \tilde \alpha = \alpha$. 
In fact the proof proceeds by desingularising a K\"ahler current with
analytic singularities in the class $\alpha$, so, using the current $T$ defined above, we see that the $\mu$-exceptional
locus maps exactly onto $\Null{\alpha}$. Up to blowing up further the exceptional locus is a SNC divisor.
By Remark \ref{remarkcoeffs} we have 
$$
\mu^* \alpha = \tilde \alpha + \sum_{j=1}^k r_j D_j.
$$
with $r_j>0$ for all $j \in \{1, \ldots, k\}$. Since $\alpha$ is nef and big, the class $\tilde \alpha + m \mu^* \alpha$
is K\"ahler for all $m >0$. Thus up to replacing the decomposition above by 
$$
\mu^* \alpha = \frac{\tilde \alpha+m \mu^* \alpha}{m+1} + \sum_{j=1}^k \frac{r_j}{m+1} D_j
$$ 
for $m \gg 0$ we can suppose that $r_j<1$ for all $j \in \{1, \ldots, k\}$.
Since $X$ is smooth we have
$K_{\tilde X} = \mu^* K_X +  \sum_{j=1}^k a_j E_j$
with $a_j$ a positive integer. Since $r_j<1$ we have $a_j-r_j>-1$ for all $E_j$ mapping onto $Z$.
Thus we can choose a $c \in \R^+$ such that $a_j- c r_j \geq -1$ for all $E_j$ mapping onto $Z$
and equality holds for at least one divisor.
\end{proof}

As a first step toward Theorem \ref{theoremmain} we can now prove the following:

\begin{theorem} \label{theoremuniruled}
Let $X$ be a compact K\"ahler manifold of dimension $n$. 
Suppose that Conjecture \ref{conjectureBDPP} holds 
for all manifolds of dimension at most $n-1$. 
Suppose that $K_X$ is pseudoeffective but not nef, and let $\omega$ be a K\"ahler class on $X$ such that $\alpha:=K_X+\omega$
is nef and big but not K\"ahler.

Let $Z \subset X$ be an irreducible component of maximal dimension
of the null-locus $\Null{\alpha}$, and let $\holom{\pi}{Z'}{Z}$ be the composition of the normalisation and a resolution
of singularities. Let $k$ be the numerical dimension of $\pi^* \alpha|_Z$ (cf. Definition \ref{definitionnumericaldimension}). Then we have
$$
K_{Z'} \cdot \pi^* \alpha|_Z^{k} \cdot \pi^* \omega|_Z^{\dim Z-k-1}  < 0.
$$
In particular $Z'$ is uniruled.
\end{theorem}

\begin{proof}[Proof of Theorem \ref{theoremuniruled}]
Since $\alpha=K_X+\omega$ and $\pi^* \alpha|_Z^{k+1}=0$ we have
$$
\pi^* K_X|_Z \cdot \pi^* \alpha|_Z^{k} = - \pi^* \omega|_Z \cdot \pi^* \alpha|_Z^{k}.
$$
By hypothesis $k < \dim Z$ so $\dim Z-k-1$ is non-negative.
Since $\pi^* \alpha|_Z^{k}$ is a non-zero nef class and $\omega$ is K\"ahler this implies by Remark \ref{remarknumericaldimension}
that
\begin{equation} \label{equationKXnegative}
\pi^* K_X|_Z \cdot \pi^* \alpha|_Z^{k} \cdot \pi^* \omega|_Z^{\dim Z-k-1} 
= - \pi^* \omega|_Z^{\dim Z-k} \cdot \pi^* \alpha|_Z^{k} < 0.
\end{equation}
Our goal will be to prove that
$$
K_{Z'} \cdot \pi^* \alpha|_Z^{k} \cdot \pi^* \omega|_Z^{\dim Z-k-1}  < 0.
$$
This inequality implies the statement: since $K_{Z'}$ is not pseudoeffective
and Conjecture \ref{conjectureBDPP} holds in dimension at most $n-1 \geq \dim Z'$ we obtain that $Z'$ is uniruled.

We will make a case distinction:

{\em Step 1. The null-locus of $\alpha$ contains an irreducible divisor.}
Since $Z$ has maximal dimension, it is a divisor.
Since $K_X$ is pseudoeffective we can consider the divisorial Zariski decomposition \cite[Defn.3.7]{Bou04}
$$
c_1(K_X) = \sum e_i Z_i + P(K_X),
$$
where $e_i \geq 0$, the $Z_i \subset X$ are prime divisors and $P(K_X)$ is a modified nef class \cite[Defn.2.2]{Bou04}. 
Arguing as in \cite[Lemma 4.1]{HP16} we see that the inequality \eqref{equationKXnegative} implies (up to renumbering) that
$Z_1=Z$ and 
\begin{equation} \label{equationZnegative}
\pi^* (c_1(\sO_Z(Z))) \cdot  \pi^* \alpha|_Z^{k} \cdot \pi^* \omega|_Z^{n-k-2} < 0. 
\end{equation}
Thus the normal bundle $N_{Z/X} \simeq \sO_Z(Z)$ is negative with respect to these nef classes. 
Moreover there exist effective $\Q$-divisors on $D_1$ and $D_2$ on $Z'$ such that
$$
K_{Z'} = \pi^* (K_X+Z) + D_1 - D_2
$$
and $\pi(D_1)$ has codimension at least two in $Z$ (cf. \cite[Prop.2.3]{Rei94}). Thus we have
$$
K_{Z'} \cdot \pi^* \alpha|_Z^{k} \cdot \pi^* \omega|_Z^{n-k-2} 
\leq \pi^* (K_X+Z) \cdot \pi^* \alpha|_Z^{k} \cdot \pi^* \omega|_Z^{n-k-2}. 
$$
Combining  \eqref{equationKXnegative} and  \eqref{equationZnegative} we obtain that the right hand side is negative.

{\em Step 2. The null-locus of $\alpha$ has no divisorial components.}
In this case we know by Lemma \ref{lemmalccentre} that there exists a $c>0$ such that $Z$ is a maximal lc centre
for $(X, c \alpha)$. The classes $\pi^* \alpha|_Z$ and $\pi^* \omega|_Z$ are nef, so by Theorem \ref{theoremweaksubadjunction} 
we have
$$
K_{Z'} \cdot \pi^* \alpha|_Z^{k} \cdot \pi^* \omega|_Z^{\dim Z-k-1} 
\leq 
\pi^* (K_X+ c \alpha)|_Z \cdot \pi^* \alpha|_Z^{k} \cdot \pi^* \omega|_Z^{\dim Z-k-1}.
$$
Since $k$ is the numerical dimension of $\pi^* \alpha|_Z$ we have $c \ \pi^* \alpha|_Z^{k+1} \cdot \pi^* \omega|_Z^{\dim Z-k-1}=0$.
Thus \eqref{equationKXnegative} yields the claim.
\end{proof}

\begin{remark}
We used the hypothesis that $Z$ has maximal dimension only in Step 1, so our proof actually yields a more precise statement:
$\Null{\alpha}$ contains a uniruled divisor or all the components of $\Null{\alpha}$ are uniruled.
\end{remark}

We come now to the technical problem mentioned in the introduction:

\begin{problem} \label{problembound}
Let $X$ be a compact K\"ahler manifold, and let $\alpha \in N^1(X)$ be a nef cohomology class. Does there exist 
a real number $b>0$ such that for every (rational) curve $C \subset X$ we have either $\alpha \cdot C = 0$ or
$\alpha \cdot C \geq b$ ?
\end{problem}

\begin{remark} \label{remarkbound}
If $\alpha$ is the class of a nef $\Q$-divisor, the answer is obviously yes: some positive multiple $m \alpha$ is integral, 
so we can choose $b:= \frac{1}{m}$. If $\alpha$ is a K\"ahler class the answer is also yes: by Bishop's theorem there
are only finitely many deformation families of curves $C$ such that $\alpha \cdot C \leq 1$, so $\alpha \cdot C$ takes 
only finitely many values in $]0, 1[$. However, even for the class of an $\R$-divisor on a projective manifold $X$ it seems
possible that the values $\alpha \cdot C$ accumulate at $0$ \cite[Rem.1.3.12]{Laz04a}. In the proof of Theorem \ref{theoremmain}
we will use that $\alpha$ is an adjoint class to obtain the existence of the lower bound $b$.
\end{remark}

The problem \ref{problembound} is invariant under certain birational morphisms:

\begin{lemma} \label{lemmainvariance}
Let \holom{\pi}{X}{X'} be a holomorphic map between normal projective varieties $X$ and $X'$.
Let $\alpha'$ be a nef $\R$-divisor class on $X'$ and set $\alpha:=\pi^* \alpha'$.

a) Suppose that there exists a real number $b>0$ such that for every (rational) curve $C' \subset X'$ 
we have $\alpha' \cdot C' = 0$ or $\alpha' \cdot C' \geq b$.  
Then for every (rational) curve $C \subset X$ we have $\alpha \cdot C = 0$ or $\alpha \cdot C \geq b$.

b) Suppose that there exists a real number $b>0$ such that for every (rational) curve $C \subset X$ 
we have $\alpha \cdot C = 0$ or $\alpha \cdot C \geq b$. Suppose also that $X$ has klt singularities and $\pi$
is the contraction of a $K_X$-negative extremal ray.
Then for every (rational) curve $C' \subset X'$ we have $\alpha' \cdot C' = 0$ or $\alpha' \cdot C' \geq b$.
\end{lemma}

\begin{proof} {\em Proof of a)}
Let $C \subset X$ be a (rational) curve such that $\alpha \cdot C \neq 0$. the image $C':=\pi(C) \subset X'$ is a (rational) curve
and the induced map $C \rightarrow C'$ has degree $d \geq 1$. Thus the projection formula yields
$$
\alpha \cdot C = \pi^* \alpha' \cdot C = \alpha' \cdot \pi_*(C) = d \alpha' \cdot C' \geq d b \geq b.
$$

{\em Proof of b)} Let $C' \subset X'$ be an arbitrary (rational) curve such that $\alpha' \cdot C' \neq 0$.
By \cite[Cor.1.7(2)]{HM07} the natural map $\fibre{\pi}{C'} \rightarrow C'$ has a section, so there exists a (rational) curve
$C \subset X$ such that the map $\pi|_{C}: C \rightarrow C'$ has degree one.  
Thus the projection formula yields
$$
\alpha' \cdot C' = \alpha' \cdot \pi_*(C) = \pi^* \alpha \cdot C \geq b.
$$
\end{proof}

\begin{remark} \label{remarkkaehlercase}
It is easy to see that statement a) also holds when $X$ and $X'$ are compact K\"ahler manifolds and 
$\alpha'$ is a nef cohomology class on $X'$.
\end{remark}

\begin{corollary} \label{corollaryinvariance}
Let $X$ be a normal projective $\Q$-factorial variety with klt singularities, and let $\alpha$ be a nef $\R$-divisor class on $X$.
Suppose that there exists a real number $b>0$ such that for every (rational) curve $C \subset X$ we have $\alpha \cdot C = 0$ 
or $\alpha \cdot C \geq b$.  Let $\merom{\mu}{X}{X'}$ be the divisorial contraction or flip of a $K_X$-negative 
extremal ray $\Gamma$ such that $\alpha \cdot \Gamma=0$. Set $\alpha':=\mu_*(\alpha)$. Then $\alpha'$ is
a nef  $\R$-divisor class on $X'$ and 
for every (rational) curve $C \subset X$ we have $\alpha \cdot C = 0$ or $\alpha \cdot C \geq b$.  
\end{corollary}

\begin{proof}
If $\mu$ is divisorial the condition $\alpha \cdot \Gamma =0$ implies that $\alpha = \mu^* \alpha'$ \cite[Cor.3.17]{KM98}. 
Thus Lemma \ref{lemmainvariance}, b) applies. 
If $\mu$ is a flip, let $\holom{f}{X}{Y}$ be the contraction of the extremal ray and $\holom{f'}{X'}{Y}$ the flipping map.
Since $\alpha \cdot \Gamma=0$ there exists an $\R$-divisor class $\alpha_Y$ on $Y$ 
such that $\alpha = f^* \alpha_Y$ \cite[Cor.3.17]{KM98}.
Moreover we have $\alpha'=(f')^* \alpha_Y$ since they coincide in the complement of the flipped locus.
Thus we conclude by applying Lemma \ref{lemmainvariance},b) to $f$ and
Lemma \ref{lemmainvariance},a) to $f'$.
\end{proof}

\begin{proposition} \label{propositionboundedpseff}
Let $F$ be a projective manifold, and let $\alpha$ be a nef $\R$-divisor class on $F$. Suppose that there exists
a real number $b>0$ such that for every rational curve $C \subset F$ such that $\alpha \cdot C \neq 0$ we have
\begin{equation} \label{conditionbound}
\alpha \cdot C>b.
\end{equation}
Then one of the following holds
\begin{itemize}
\item $F$ is dominated by rational curves $C \subset F$ such that $\alpha \cdot C=0$; or
\item the class $K_F+ \frac{2 \dim F}{b} \alpha$ is pseudoeffective.
\end{itemize}
\end{proposition}

\begin{proof}
Note that, up to replacing $\alpha$ by $\frac{2 \dim F}{b} \alpha$, we can suppose that 
\begin{equation} \label{conditionbound2}
\alpha \cdot C>2 \dim F
\end{equation}
for every rational curve $C \subset F$ that is not $\alpha$-trivial.
Suppose that $K_F+\alpha$ is not pseudoeffective, then our goal is to show that $F$ is covered by $\alpha$-trivial rational curves.
Since $K_F+\alpha$ is not pseudoeffective, there exists an ample $\R$-divisor $H$ such that $K_F+\alpha+H$ is not pseudoeffective.
Since $H$ and $\alpha+H$ are ample we can 
choose effective $\R$-divisors $\Delta_H \sim_\R H$ and $\Delta \sim_\R \alpha+H$ such that the pairs $(F, \Delta_H)$
and $(F, \Delta)$ are klt.
 By \cite[Cor.1.3.3]{BCHM10} we can run a $K_F+\Delta$-MMP 
$$
(F, \Delta)=: (F_0, \Delta_0) \stackrel{\mu_0}{\dashrightarrow} (F_1, \Delta_1)  \stackrel{\mu_1}{\dashrightarrow}
\ldots  \stackrel{\mu_k}{\dashrightarrow} (F_k, \Delta_k),  
$$
that is for every $i \in \{0, \ldots, k-1\}$ the map
$\merom{\mu_i}{F_i}{F_{i+1}}$ is either a divisorial Mori contraction of a
$K_{F_i}+\Delta_i$-negative extremal ray $\Gamma_i$ in $\NE{X_i}$  or the flip of a small contraction 
of such an extremal ray. 
Note that for every $i \in \{0, \ldots, k \}$ the variety $F_i$  is normal $\Q$-factorial and the pair $(F_i, \Delta_i)$ is klt.
Moreover $F_k$ admits a Mori contraction of fibre type \holom{\psi}{F_k}{Y} contracting an extremal ray $\Gamma_k$
such that $(K_{F_k}+\Delta_k) \cdot \Gamma_k < 0$.

Set $\Delta_{H,0}:=\Delta_H, \alpha_0 :=\alpha$ and for all $i \in \{0, \ldots, k-1\}$ we define inductively
$$
\Delta_{H, i+1}:= (\mu_i)_* (\Delta_{H,i}), \ \alpha_{i+1} := (\mu_i)_* (\alpha_i).
$$
Note that for all $i \in \{0, \ldots, k\}$ we have
\begin{equation} \label{twopairs}
K_{F_i} + \Delta_i \equiv K_{F_i} + \Delta_{H,i} + \alpha_i. 
\end{equation}
We claim that for all $i \in \{0, \ldots, k\}$ the $\R$-divisor class $\alpha_i$ is nef and $\alpha_i \cdot \Gamma_i=0$.
Moreover the pairs $(X_i, \Delta_{H,i})$ are klt. Assuming this for the time being, let us see how to conclude:
since \holom{\psi}{F_k}{Y} is a Mori fibre space and the extremal ray $\Gamma_k$ is $\alpha_k$-trivial, we see that
$F_k$ is dominated by $\alpha_k$-trivial rational curves $(C_t)_{t \in T}$. 
A general member of this family of rational curves is not contained in the exceptional locus of $F_0 \dashrightarrow F_k$,
so the strict transforms define a dominant family of rational curves $(C_t')_{t \in T}$ of $F_0$.
Since all the birational contractions in the MMP $F_0 \dashrightarrow F_k$ are $\alpha_\bullet$-trivial, 
we easily see (cf. the proof of Corollary \ref{corollaryinvariance}) that
$$
\alpha \cdot C_t' = \alpha_k \cdot C_t =0.
$$

{\em Proof of the claim.} Since $\alpha_0$ is nef, we have
$$
0 > (K_{F_0} + \Delta_0) \cdot \Gamma_0 = (K_{F_0} + \Delta_{H,0} + \alpha_0) \cdot \Gamma_0 \geq (K_{F_0} + \Delta_{H,0})  
\cdot \Gamma_0.
$$
Thus the extremal ray $\Gamma_0$ is $K_{F_0} + \Delta_{H,0}$-negative, in particular the pair 
$(F_1, \Delta_1)$ is klt \cite[Cor.3.42, 3.43]{KM98}.
Moreover there exists by \cite[Thm.1]{Kaw91} a rational curve $[C_0] \in \Gamma_0$ such that
$(K_{F_0}+\Delta_{H,0}) \cdot C_0 \geq - 2 \dim F$. Thus if $\alpha_0 \cdot C_0 \neq 0$, the inequality \eqref{conditionbound2} implies that
$$
(K_{F_0} + \Delta_0) \cdot C_0 =  (K_{F_0} + \Delta_{H,0})  \cdot C_0 + \alpha_0 \cdot C_0
> 0.
$$
In particular the extremal ray $\Gamma_0$ is not $K_{F_0} + \Delta_0$-negative, a contradiction to our assumption. Thus
we have $\alpha_0 \cdot C_0 = 0$. By Corollary \ref{corollaryinvariance} this implies that $\alpha_1$ is nef and satisfies
the inequality \eqref{conditionbound2}. The claim now follows by induction on $i$.
\end{proof}

\begin{remark} \label{remarkMRC}
For the proof of Theorem \ref{theoremmain} we will use the MRC fibration of a uniruled manifold.
Since the original papers \cite{KMM92, Cam92} are formulated for projective manifolds, let us recall that for
a compact K\"ahler manifold $M$ that is uniruled the MRC fibration is defined as an almost holomorphic
map $\merom{f}{M}{N}$ such that the general fibre $F$ is rationally connected and the dimension of $F$ is maximal
among all the fibrations of this type. The existence of the MRC fibration follows, as in the projective case, from the existence 
of a quotient map for covering families \cite{Cam04b}.
The base $N$ is not uniruled : arguing by contradiction we consider a dominating family 
$(C_t)_{t \in T}$ of rational curves on $N$. Let $M_t$ be a desingularisation of $\fibre{f}{C_t}$ for a general $C_t$, then $M_t$ is a compact K\"ahler manifold with a fibration onto a curve $M_t \rightarrow C_t$ such that the general fibre is rationally connected.
In particular $H^0(M_t, \Omega_{M_t}^2)=0$ so $M_t$ is projective by Kodaira's criterion. Thus we can apply the Graber-Harris-Starr
theorem \cite{GHS03} to see that $M_t$ is rationally connected, a contradiction.
\end{remark}

\begin{proof}[Proof of Theorem \ref{theoremmain}]
Let $\omega$ be a K\"ahler class such that $\alpha:= K_X+\omega$ is nef and big, but not K\"ahler. By Theorem \ref{theoremuniruled}
there exists a subvariety $Z \subset X$ contained in the null-locus $\Null{\alpha}$ that is uniruled.
More precisely let $\pi: Z' \rightarrow Z$ be a desingularisation, and denote by $k$ the numerical dimension of $\alpha':=\pi^* \alpha|_Z$.
Then we know by Theorem \ref{theoremuniruled} that
$$
K_{Z'} \cdot \alpha'^{k} \cdot \pi^* \omega|_Z^{\dim Z-k-1} < 0. 
$$
Since $\alpha'^{k+1}=0$ this actually implies that
\begin{equation} \label{alwaysnegative}
(K_{Z'}+\lambda \alpha')  \cdot \alpha'^{k} \cdot \pi^* \omega|_Z^{\dim Z-k-1} < 0 \qquad \forall \ \lambda>0.
\end{equation}
Our goal is to prove that this implies that $Z$ contains a $K_X$-negative rational curve. Arguing by contradiction we  
suppose that $K_X \cdot C \geq 0$ for every rational curve $C \subset Z$. Since $\omega$ is a K\"ahler class this 
implies by Remark \ref{remarkbound} that there exists a $b>0$ such that for every rational curve $C \subset Z$ we have
\begin{equation} \label{positiveonrational}
\alpha \cdot C = (K_X+\omega) \cdot C \geq \omega \cdot C \geq b.
\end{equation}
By Lemma \ref{lemmainvariance}a) and Remark \ref{remarkkaehlercase} this implies that for
every rational curve $C' \subset Z'$ we have $\alpha' \cdot C' = 0$ or $\alpha' \cdot C' \geq b$.

Since $Z'$ is uniruled we can consider the MRC-fibration $\merom{f}{Z'}{Y}$ (cf. Remark \ref{remarkMRC}).
The general fibre $F$ is rationally connected, in particular we can consider $\alpha'|_F$ as a nef $\R$-divisor class. Moreover the inequality above shows that $\alpha'|_F$ 
satisfies the condition \eqref{conditionbound} in Proposition \ref{propositionboundedpseff}.
If $F$ is dominated by $\alpha'|_F$-trivial rational curves, then $Z'$ is dominated by $\alpha'$-trivial rational curves.
A general member of this dominating family is not contracted by $\pi$, so $Z$ is dominated by $\alpha$-trivial rational curves.
This possibility is excluded by \eqref{positiveonrational}, so 
Proposition \ref{propositionboundedpseff} shows that there exists 
a $\lambda>0$ such that $K_F+\lambda \alpha'|_F$ is pseudoeffective. 

We will now prove that $K_{Z'}+\lambda \alpha$ is pseudoeffective, which clearly contradicts \eqref{alwaysnegative}.
If $\holom{\nu}{Z''}{Z}$ is a resolution of the indeterminacies of $f$ such that $K_{Z''}+\nu^* (\lambda \alpha)$ is pseudoeffective,
then $K_{Z'} + \lambda \alpha = (\nu)_* (K_{Z''}+\nu^* (\lambda \alpha))$ is pseudoeffective. Thus we can assume without loss
of generality that the MRC-fibration $f$ is a holomorphic map.
Let $\omega'$ be a K\"ahler class on $Z'$, then
for every $\varepsilon>0$ the class $\lambda \alpha'+\varepsilon \omega$ is K\"ahler and 
$K_F+(\lambda \alpha+\varepsilon \omega)|_F$ is pseudoeffective. Thus we can apply Theorem \ref{theoremdirectimage2}
to $\holom{f}{Z'}{Y}$ to see that
$$
K_{Z'/Y}+\lambda \alpha+\varepsilon \omega
$$
is pseudoeffective. Note now that $Y$ has dimension at most $\dim X-2$ is not uniruled (Remark \ref{remarkMRC})
Since we assume that Conjecture \ref{conjectureBDPP}
holds in dimension up to $\dim X-1$, we obtain that $K_Y$ is pseudoeffective. Thus we see that
$K_{Z'}+\lambda \alpha+\varepsilon \omega$ is pseudoeffective for all $\varepsilon>0$. 
The statement follows by taking the limit $\varepsilon \to 0$. 
\end{proof}

\def\cprime{$'$}

\end{document}